\documentclass[journal]{IEEEtran}

\usepackage{graphicx}
\usepackage{booktabs}
\usepackage{tabularx}
\usepackage{refcount}
\usepackage{cite}
\usepackage{url}
\usepackage{amsmath}
\usepackage{amssymb}
\usepackage{enumitem}
\usepackage{bm}
\usepackage{algorithm,algorithmic}

\usepackage{hyperref}
\hypersetup{
  colorlinks   = true, %Colours links instead of ugly boxes
  urlcolor     = blue, %Colour for external hyperlinks
  linkcolor    = blue, %Colour of internal links
  citecolor   = red %Colour of citations
}
\usepackage[capitalise]{cleveref}
\crefname{equation}{}{}
\crefname{assum}{Assumption}{Assumptions}
\crefname{theo}{Theorem}{Theorems}
\usepackage{autonum}
\usepackage{color}\definecolor{RED}{rgb}{1,0,0}\definecolor{BLUE}{rgb}{0,0,1}
\usepackage[caption=false,subrefformat=parens,labelformat=parens]{subfig}
\DeclareMathOperator*{\minimize}{minimize}

\DeclareMathOperator*{\st}{subject\ to}
\DeclareMathOperator*\argmin{\text{argmin}}

\def\citet #1{\cite{#1}}
\def\maxn #1{\|#1\|_{\text{max}}}

\def\twon #1{\|#1\|}

\def\tr{\text{tr}}
\def\ln{\text{ln}}

\newcommand{\mathletter}[1]{%
	\expandafter\newcommand\csname b#1\endcsname{\mathbb #1}
	\expandafter\newcommand\csname c#1\endcsname{\mathcal #1}
	\expandafter\newcommand\csname f#1\endcsname{\mathfrak #1}
	\expandafter\newcommand\csname til#1\endcsname{\widetilde #1}
	\expandafter\newcommand\csname ha#1\endcsname{\widehat #1}
	\expandafter\newcommand\csname bf#1\endcsname{\bf #1}
	\expandafter\newcommand\csname s#1\endcsname{\mathsf #1}
}%
\def\mathletters#1{\mathlettersB #1,,}
\def\mathlettersB#1,{\ifx,#1,\else\mathletter #1\expandafter\mathlettersB\fi}
\mathletters{A,B,C,D,E,F,G,H,I,J,K,L,M,N,O,P,Q,R,S,T,U,V,W,X,Y,Z}

\newcommand{\mathletterl}[1]{%
	\expandafter\providecommand\csname v#1\endcsname{\vec{#1}}
}%
\def\mathlettersl#1{\mathlettersC #1,,}
\def\mathlettersC#1,{\ifx,#1,\else\mathletterl #1\expandafter\mathlettersC\fi}
\mathlettersl{a,b,c,d,e,f,g,h,i,j,k,l,m,n,o,p,q,r,s,t,u,v,w,x,y,z}

\def\bea{\begin{equation}\begin{alignedat}{-1}}
\def\ena{\end{alignedat}\end{equation}}
\def\bee{\begin{equation}}
\def\ene{\end{equation}}

\renewcommand{\vec}[1]{\mathbf{#1}}

\newtheorem{theo}{Theorem}
\newtheorem{lemma}{Lemma}
\newtheorem{assum}{Assumption}

\newtheorem{remark}{Remark}

\newtheorem{example}{Example}
\newenvironment{proof}{\begin{IEEEproof}}{\end{IEEEproof}}

\def\T{\mathsf{T}}

\def\bone{{\mathbf{1}}}
\def\bzero{{\mathbf{0}}}

\newcommand{\p}{p}

\begin{document}

\title{Innovation Compression for Communication-efficient Distributed Optimization with Linear Convergence}
\author{Jiaqi~Zhang,
	Keyou~You,~\IEEEmembership{Senior Member,~IEEE,}
	and~Lihua~Xie,~\IEEEmembership{Fellow,~IEEE} % <-this % stops a space
	\thanks{Jiaqi Zhang and Keyou You  are with the Department of Automation, Beijing National Research Center for Information Science and Technology, Tsinghua University, Beijing 100084, China. E-mail: zjq16@mails.tsinghua.edu.cn, youky@tsinghua.edu.cn (Corresponding author: Keyou You). }
	\thanks{Lihua Xie is with the School of Electrical and Electronic Engineering, Nanyang Technological University, Singapore
	639798, Singapore. e-mail: elhxie@ntu.edu.sg.}
}
\maketitle

\begin{abstract}
Information compression is essential to reduce communication cost in distributed optimization over peer-to-peer networks. This paper proposes a {\em co}mmunication-efficient {\em l}inearly convergent {\em d}istributed (COLD) algorithm to solve strongly convex optimization problems. By compressing innovation vectors, which are the differences between decision vectors and their estimates, COLD is able to achieve linear convergence for a class of $\delta$-contracted compressors. We explicitly quantify how the compression affects the convergence rate and show that COLD matches the same rate of its uncompressed version. To accommodate a wider class of compressors that includes the binary quantizer, we further design a novel dynamical scaling mechanism and obtain the linearly convergent Dyna-COLD. Importantly, our results strictly improve existing results for the quantized consensus problem. Numerical experiments demonstrate the advantages of both algorithms under different compressors.
\end{abstract}


\section{Introduction}

We consider the following distributed optimization problem over a network of $n$ interconnected nodes
\bee\label{original}
\minimize_{\vx\in\bR^d}\ f(\vx)\triangleq\sum_{i=1}^n f_i(\vx),
\ene
where the local function $f_i$ is only known to node $i$. Each node aims to find an optimal solution $\vx^\star\in\argmin f(\vx)$ by communicating with only a subset of nodes that are defined as its neighbors. Such a distributed model has been shown to achieve promising results in formation control \cite{li2018dynamic}, multi-agent consensus \cite{li2013consensus}, networked resource allocation \cite{zhang2020distributed}, and empirical risk minimization problems \cite{zhang2019decentralized}.

Till now, many novel algorithms have been proposed by transmitting messages with infinite precision, such as DGD \cite{nedic2009distributed}, EXTRA \cite{shi2015extra}, NIDS \cite{li2019decentralized}, DIGing \cite{nedic2017achieving,qu2017harnessing}, to name a few.  As nodes iteratively communicate with neighbors, communication cost can be a bottleneck for the efficiency of distributed optimization. 

To resolve it, a compressor $\sQ(\cdot):\bR^d\rightarrow\bR^d$ is unavoidable for each node to send compressed messages which can be usually encoded with a moderate number of bits. For example, the output of a binary quantizer (i.e. $\sQ(x)=1$ if $x\geq 0$ and $\sQ(x)=-1$ otherwise) needs only one bit to encode. Since $\sQ(\cdot)$ is typically highly nonlinear and nonsmooth, how to design a provably convergent distributed algorithm with efficient communication compressions has attracted an increasing attention. To this end, quantized variants of the DGD \cite{nedic2009distributed} have been proposed for convex problems \citet{nedic2008distributed,carli2010gossip,reisizadeh2019exact,li2020computation,yi2014quantized,doan2020convergence} and non-convex problems \citet{tang2018communication,lu2020moniqua,taheri2020quantized}, respectively. Unfortunately, their convergence rates are constrained by the sublinear rate of DGD. Though linear convergence has been achieved in \cite{lee2018finite,magnusson2020maintaining,xiong2021quantized},  they are restricted to a specified compressor which is then extended to a class of stochastic compressors in \citet{li2021compressed,liu2020linear,kovalev2020linearly}. However, their stochastic compressors are assumed to be unbiased and $\delta$-contracted  in the sense that $\bE[\sQ(\vx)]=\vx$ and $\bE\|\sQ(\vx)-\vx\|^2\leq\delta\|\vx\|^2,\ \forall \vx\in\bR^d$. Clearly, this condition excludes some important compressors, including deterministic compressors such as the binary quantizer. In fact, the unavoidable round-off error in finite-precision implementation can be regarded as the consequence of a biased compressor.

In this paper, we propose two communication-efficient distributed algorithms, namely COLD and Dyna-COLD, that converge linearly under two important classes of compressors. In comparison, our contributions can be summarized as follows:
\begin{itemize}[leftmargin=*,topsep=0pt,noitemsep]
	\item We notice that it is efficient to compress the innovation---the difference between a decision variable and its estimate, as it is expected to eventually decrease to zero. By leveraging such an observation, COLD is the first distributed algorithm that achieves linear convergence for $\mu$-strongly convex and $L$-smooth functions with both unbiased and biased $\delta$-contracted compressors, which is in sharp contrast to \cite{nedic2008distributed,carli2010gossip,reisizadeh2019exact,li2020computation,yi2014quantized,doan2020convergence,lee2018finite,magnusson2020maintaining,xiong2021quantized,liu2020linear,kovalev2020linearly,li2021compressed} as they are only applicable to some specific biased or unbiased compressors.

	\item The effect of compression on the convergence rate of COLD is explicitly revealed, e.g., the number of iterations to find an $\epsilon$-optimal solution is $\cO\big(\max\{\frac{L}{\mu},\frac{1}{(1-\delta)^2\rho}\}\log\frac{1}{\epsilon}\big)$ for unbiased $\delta$-contracted compressors, where $\delta<1$ characterizes the compression resolution and $\rho<1$ is the spectral gap of the network. Specifically, $\delta=0$ corresponds to the uncompressed case under which COLD matches the rate of NIDS \citet{li2019decentralized} with infinite precision.

	\item By designing a novel dynamical scaling mechanism, we obtain linearly convergent Dyna-COLD  for a wider class of compressors, which covers the quantizers  in \cite{lee2018finite,magnusson2020maintaining,xiong2021quantized}. Numerical experiments validate the linear convergence of Dyna-COLD even for the 1-bit binary quantizer and show its advantages over competitors.

	\item When \eqref{original} is reduced to the consensus problem, i.e., $f_i(\vx)=\|\vx-\vx_i\|^2$, we strictly improve the rate of the state-of-the-art CHOCO-GOSSIP method \cite{koloskova2019decentralized} from $\cO(\frac{1}{(1-\delta)\rho^2}\log\frac{1}{\epsilon})$ to $\cO(\frac{1}{(1-\delta)\rho}\log\frac{1}{\epsilon})$, which matches the best rate in \citet{xiao2004fast}. By the dynamic scaling, we generalize and strengthen the theoretical result in \cite{li2011distributed} which was established for a finite-level uniform quantizer. 
	\end{itemize}

The rest of this paper is organized as follows. \cref{sec3} formalizes the problem and introduces two important classes of compressors. \cref{sec4} develops the COLD  and provides its theoretical results. \cref{sec4b} introduces the Dyna-COLD with the dynamic scaling mechanism and shows its linear convergence rate. \cref{sec5} validates the theoretical finding via numerical experiments. Some concluding remarks are drawn in \cref{sec6}. A preliminary version of this paper has been submitted to CDC \cite{linearly2021zhang}, which covers only one class of compressors and does not include the dynamic scaling method.

{\bf Notation:} We use $\nabla f$ to denote the gradient of $f$. $[A]_{ij}$ denotes the $(i,j)$-th element of $A$. $\lambda_j(A)$ denotes the $j$-th largest eigenvalue of $A$. $\|\cdot\|_{\p}$ denotes the vector $p$-norm. Let $\langle X,Y\rangle_M\triangleq\tr(X^\T M Y)$, $\|A\|_M^2\triangleq\langle A,A\rangle_M$ and $\|A\|^2\triangleq\|A\|_I^2$, where $M$ is a scalar or a symmetric matrix, and $I$ is the identity matrix. Define the matrix norm $\|A\|_\text{max}\triangleq \max_{i}\|\va_{i\cdot}\|_{\p}$, where $A\in\bR^{n\times d}$ and $\va_{i\cdot}$ is the $i$-th row of $A$. $\bone$ denotes a vector with all ones. $\cO(\cdot)$ denotes the big-O notation. We say a sequence $\{x^k\}$ converges linearly to an optimal solution $x^\star$ if $\|x^k-x^\star\|^2=\cO((1-\alpha)^k)$ for some $\alpha\in(0,1]$, which implies that an $\epsilon$-optimal solution can be obtained in $\widehat k=\cO(\frac{1}{\alpha}\log\frac{1}{\epsilon})$ iterations, i.e., $\|x^{k}-x^\star\|^2\leq\epsilon$ for all $k\geq \widehat k$.

\section{Problem Formulation}\label{sec3}

\subsection{Compressors}

A compressor $\sQ(\cdot):\bR^d\rightarrow\bR^d$ is a (possibly stochastic) mapping either for quantization or sparsification, and its output can be usually encoded with much fewer bits than its input. Instead of focusing on a specific compressor, this work considers the following two broad classes of compressors.
\begin{assum}\label{assum1}
	For some $\delta\in[0,1)$, $\sQ$ satisfies that
	\bee
	\bE[\|\sQ(\vx)-\vx\|^2]\leq\delta\|\vx\|^2,\ \forall \vx\in\bR^d
	\ene
	where $\bE[\cdot]$ denotes the expectation over $\sQ$.
\end{assum}

\cref{assum1} requires the mean square of the \emph{relative} compression error to be bounded, which is  called $\delta$-contracted compressors in \citet{tang2018communication,koloskova2019decentralized,liu2020linear,taheri2020quantized}. 
\begin{assum}\label{assum3}
	For some $\delta\in[0,1)$, $\sQ$ satisfies that
	\bea
	\|\sQ(\vx)-\vx\|_{\p}\leq\delta,\ \forall \vx\in\{\vx\in\bR^d:\|\vx\|_{\p}\leq 1\}.
	\ena
\end{assum}

In \cref{assum3}, the \emph{absolute} compression error over the unit ball is bounded by $\delta$. Since we do not impose any condition on $\sQ(\vx)$ for $\|\vx\|_{\p}>1$, \cref{assum3} is strictly weaker than the deterministic version of \cref{assum1}, and covers the compressors in \cite{nedic2008distributed,carli2010gossip,reisizadeh2019exact,li2020computation,yi2014quantized,doan2020convergence,lee2018finite,magnusson2020maintaining,xiong2021quantized}. To our best knowledge, Moniqua \cite{lu2020moniqua} is the only distributed method that converges sublinearly under a special case $p=\infty$ of \cref{assum3}. In this work, we design a dynamical scaling mechanism to achieve linear convergence in \cref{sec4b}. 

A compressor $\sQ$ is \emph{unbiased} if $\bE[\sQ(\vx)]=\vx$ for all $\vx\in\bR^d$. Though unbiasedness is essential to the linear convergence in \citet{liu2020linear,kovalev2020linearly,li2021compressed}, the widely-used deterministic quantizers are  biased.  Round-off errors are inevitable in a finite-precision processor and result in a biased compressor, which satisfies both assumptions when $\vx$ is large, but only satisfies \cref{assum3} if $\vx$ lies in the subnormal range (e.g. $|\vx|<2^{-126}$ for single-precision, IEEE 754 standard \cite{ieee754}). In this view, biased compressors have more practical significance. Some commonly used compressors are given below.

\begin{example}[Compressors]\label{exap_compressor}
	\begin{enumerate}[leftmargin=*,label=(\alph*),noitemsep,nolistsep]
		\item \textbf{Nearest neighbor rounding:}\label{compressor1} Given a set of numbers $\cQ$, $\sQ(\vx)$ returns the nearest number in $\cQ$ for each coordinate of $\vx$, i.e., $[\sQ(\vx)]_i=\argmin_{x\in\cQ}{|[\vx]_i-x|}$. The uniform quantizers \cite{li2011distributed} and logarithmic quantizers \cite{elia2001stabilization} fall in this category by letting $\cQ$ be the set of numbers spaced evenly on a linear scale and a log scale, respectively. Typically, this class of compressors are biased, and may satisfy only \cref{assum3}. Particularly, the 1-bit binary quantizer, i.e., $\sQ(x)=0.5$ for $x\geq 0$ and $\sQ(x)=-0.5$ for $x< 0$, satisfies \cref{assum3} with $\delta=0.5$ and $p=\infty$.
		\item \textbf{Unbiased stochastic quantization:}\label{compressor2}  Let $\text{sgn}(\cdot)$ and $|\cdot|$ be the element-wise sign function and absolute function, respectively. The compressor is given by
		      $
		      \sQ_u(\vx)=\frac{\|\vx\|_p\cdot\text{sgn}(\vx)}{u}\circ\left\lfloor\frac{u|\vx|}{\|\vx\|_p}+\xi\right\rfloor,
		      $
		      where $\circ$ denotes the Hadamard product, $u\geq 1$, $\xi$ is a random vector uniformly sampled from $[0,1]^d$.
		      This compressor is unbiased and satisfies \cref{assum1} with $\delta={d}/{4u^2}$ for $p=2$ \cite{liu2020linear,alistarh2017qsgd}, and meets \cref{assum3} with $\delta=\frac{1}{u}\sqrt[\leftroot{-2}\uproot{2}p]{d}$. A common choice is $u=2^{l-1}$ for some integer $l\geq 1$, where each coordinate can be encoded with $l+1$ bits. Transmitting $\sQ_u(\vx)$ needs $(l+1)d+b$ bits if a scalar can be transmitted with $b$ bits with sufficient precision.
		\item \textbf{Biased quantization:} A proper scaling can reduce $\delta$ but introduce bias. In particular,  $\sQ_b(\vx)\triangleq\sQ_u(\vx)/\phi$ with $\phi> 1$ is biased. It satisfies \cref{assum1} with $\delta=1-\phi^{-1}$ for $\phi= 1+{d}/{(4u^2)}$ and $p=2$ \cite{koloskova2019decentralized}, and satisfies \cref{assum3} with $\delta=\frac{1}{u+\sqrt[\leftroot{-2}\uproot{2}p]{d}}\sqrt[\leftroot{-2}\uproot{2}p]{d}$ for $\phi= 1+{1}/{u}\cdot\sqrt[\leftroot{-2}\uproot{2}p]{d}$ and $\xi=\bone$. \label{compressor3}
		\item \textbf{Sparsification:}  Randomly selecting $l\leq d$ coordinates of $\vx$ or selecting the largest $l$ coordinates in magnitude gives a biased compressor, which meets \cref{assum1} with $\delta=1-{l}/{d}$ \cite{stich2018sparsified}, and \cref{assum3} with $\delta=(1-{l}/{d})^{1/p}$. \label{compressor4}
		\item \textbf{Compression for multimedia data:} The structure of $\vx$ can be exploited for efficient compression. For example, if $\vx$ can be viewed as an image, then the standard JPEG method is a good compressor which is biased with an adjustable compression error $\delta$ \cite{minguillon2001jpeg}.
	\end{enumerate}
\end{example}
\subsection{Cost functions and the communication network}
To achieve linear convergence, we focus on strongly convex problems of \cref{original} in this work.
\begin{assum}\label{assum4}
	Each local cost function $f_i$ is $\mu$-strongly convex and $L$-Lipschitz smooth, where $0<\mu\leq L$. That is, for all $\vx,\vy\in\bR^d,i\in\cV$,
	\bea\label{eq_lipschitz}
	&f_i(\vy)\geq f_i(\vx)+\langle \nabla f_i(\vx),\vy-\vx \rangle +\frac{\mu}{2}\|\vx-\vy\|^2,\\
	&\|\nabla f_i(\vx)-\nabla f_i(\vy)\|\leq L\|\vx-\vy\|.
	\ena
\end{assum}

Under \cref{assum4}, Problem \cref{original} has a unique optimal solution $\vx^\star\triangleq \argmin f(\vx)$ \cite{boyd2004convex}.

The interaction between nodes is modeled as an undirected network $\cG=(\cV,\cE, W)$, where $\cV=\{1,2,\cdots,n\}$ is the set of nodes, $\cE\subseteq\cV\times\cV$ is the set of edges, and edge $(i,j)\in\cE$ if and only if nodes $i$ and $j$ can directly exchange compressed messages with each other. $\cG$ is called connected if there exists a path between any pair of nodes. We denote the set of nodes that directly communicate with node $i$ as the neighbors of $i$, i.e., $\cN_i=\{j|(j,i)\in\cE\}\cup\{i\}$. $W\in\bR^{n\times n}$ is an adjacency matrix of $\cG$, i.e., $[W]_{ij}=0$ if $(i,j)\notin\cE$. It captures the interactions between neighboring nodes, e.g., $WX$ can be evaluated via local communications, where $X=[\vx_1,\cdots,\vx_n]^\T$ and $\vx_i$ denotes the local copy of node $i$.  We make the following standard assumption  \cite{shi2015extra,li2019decentralized}:
\begin{assum}\label{assum2}
	 $\cG$ is connected, and $W$ satisfies that
\begin{enumerate}[leftmargin=*,label=(\alph*),noitemsep,nolistsep]
		\item (Symmetry)  $W=W^\T$;
		\item (Consensus property)  $\text{null}(I-W)=\text{span}(\bone_n)$;
		\item (Spectral property)  $-I\prec W\preceq I$.
	\end{enumerate}
\end{assum}

An eligible $W$ can be constructed by the Metropolis rule \cite{shi2015extra,xiao2004fast}.  \cref{assum2} implies that $-1<\lambda_n(W)\leq\lambda_2(W)<1$. We define the spectral gap $\rho=1-\lambda_2(W)\in[0,1)$ to characterize the effect of $W$ on the convergence rate. Under \cref{assum2}, Problem \cref{original} is equivalent to the following optimization problem
\bea\label{original2}
&\minimize_{\vec x\in\bR^d}\ F(X)\triangleq\sum_{i=1}^n f_i(\vx_i),\\
&\st\ (I-W)X=0.
\ena
For simplicity, let $$\nabla F(X)=[\nabla f_1(\vx_1),\cdots,\nabla f_n(\vx_n)]^\T\in\bR^{n\times d}$$ and we slightly abuse the notation by letting $\sQ(X)=[\sQ(\vx_1),\cdots,\sQ(\vx_n)]^\T\in\bR^{n\times d}$ where  $X=[\vx_1,\cdots,\vx_n]^\T\in\bR^{n\times d}$. Then, $\bE\|\sQ(X)-X\|^2\leq\delta\|X\|^2$ if $\sQ$ satisfies \cref{assum1}, and $\|\sQ(X)-X\|_\text{max}\leq\delta$ if $\|X\|_\text{max}\leq 1$ and $\sQ$ satisfies \cref{assum3}.

\section{The communication-efficient COLD}\label{sec4}
This section focuses on $\delta$-contracted compressors under \cref{assum1}, where we propose COLD and explicitly derive its linear convergence rate. The novelty of COLD lies in compressing an \emph{innovation} vector, which is the discrepancy between the true value of a decision vector and its estimate. Note that innovation is an important concept in the Kalman filtering theory where it represents the discrepancy between the measurements vector and their optimal predictor \cite{anderson2012optimal}. Since the innovation is expected to asymptotically decrease to $0$, the compression error vanishes eventually, which is key to the development of COLD. We elaborate this idea by firstly revisiting the distributed consensus problem and provide a strictly improved rate over CHOCO-GOSSIP \cite{koloskova2019decentralized}.

\subsection{Compressed consensus with an explicitly improved convergence rate}

Distributed consensus is an important special case of Problem \cref{original} with $f_i(\vx)=\|\vx-\vx_i\|^2,\forall i\in\cV$, where nodes are supposed to converge to the average $\bar\vx=\frac{1}{n}\sum_{i=1}^n\vx_i$. A celebrated algorithm to solve it is proposed in \citet{xiao2004fast} with the update rule $\vx_i^{k+1}= \sum_{j\in\cN_i}[W]_{ij}\vx_j^k, \forall i\in\cV$ and $\vx_i^0=\vx_i$. Let $X^k=[\vx_1^k,\cdots,\vx_n^k]^\T$. Its compact form is given as
\bee\label{alg_con3}
X^{k+1}=WX^k.
\ene
Under \cref{assum2}, it follows from \citet{xiao2004fast} that $\vx_i^k$ of \cref{alg_con3} converges to $\bar\vx$ at a linear rate $\cO(\frac{1}{\rho}\log\frac{1}{\epsilon})$\footnote{Such a rate actually corresponds to a faster version $X^{k+1}=\frac{1}{2}(I+W)X^k$ if $|\lambda_n(W)|>\frac{1}{2}(1+|\lambda_2(W)|)$.} to find an $\epsilon$-solution.

However, it requires to transmit $\vx_i^k$ with infinite precision communication. To reduce communication cost by only transmitting $\sQ(\vx_i^k)$, Ref. \citet{aysal2008distributed} proposes $X^{k+1}=W\sQ(X^k)$, which was then improved in \citet{carli2010gossip} as $X^{k+1}=X^k+(W-I)\sQ(X^k)$.   Since $\vx_i^k$ does not converge to zero, we cannot expect the compression error $\sQ(\vx_i^k)-\vx_i^k$ to vanish. Thus, both algorithms cannot converge to $\bar\vx$.

The first linearly convergent algorithm with a $\delta$-contracted compressor is proposed in \citet{koloskova2019decentralized} as follows:
\bea\label{alg_con1}
\widehat X^{k+1}&=\widehat X^k+\sQ(X^k-\widehat X^k),\\
X^{k+1}&=X^k+\gamma(W-I)\widehat X^{k+1},
\ena
where $\gamma$ is a tunable stepsize and $\vx_i^0=\vx_i, \widehat \vx_i^0=0$. We owe the success of \cref{alg_con1} to the compression on the innovation $X^k-\widehat X^k$, where $\widehat X^k$ can be viewed as an estimate of $X^k$, and reduces to $X^k$ if the compression error is zero. Since the fixed points of $\widehat\vx_i^k$ and $\vx_i^k$ in \cref{alg_con1} are both $\bar\vx$, the innovation as well as its compression error is expected to vanish if nodes have achieved consensus. In this view, we obtain a strictly better convergence rate for \cref{alg_con1} than that of \citet{koloskova2019decentralized} in the following theorem, the proof of which is relegated to \cref{adx1}. 

\begin{theo}\label{theo2}
	Let \cref{assum1,assum2} hold, $\{X^k\}$ and $\{\widehat X^k\}$ be generated by \cref{alg_con1}, $\gamma\in\big(0,\frac{1-\delta}{(1+\delta)(1-\lambda_n(W))}\big)$, and $\ve^k=\bE\big(q\|X^k-\bone\bar \vx^\T\|^2+\|\widehat X^{k+1}- X^k\|^2\big)$, where $q=\frac{(1+\gamma(1-\lambda_n(W)))\delta}{1-\gamma(1-\lambda_n(W))}>0$. Then, it holds that
	\bee
	\ve^{k+1}\leq \sigma \ve^{k} 
	\ene
	where $\sigma=\max\big\{1-\frac{2\gamma\rho}{1+\gamma(1-\lambda_n(W))},\ \frac{\delta(1+\gamma(1-\lambda_n(W)))}{1-\gamma(1-\lambda_n(W))}\big\}<1$. In particular, if $\gamma=\frac{1-\delta}{(3+\delta)(1-\lambda_n(W))}$, then 
	\bee
	\sigma=1-\frac{(1-\delta)\rho}{2(1-\lambda_n(W))}.
	\ene
\end{theo} 

\begin{remark}
	\cref{theo2} shows that each node in \cref{alg_con1} exactly converges to the average at the linear rate $\cO(\frac{1}{(1-\delta)\rho}\log\frac{1}{\epsilon})$, which strictly improves the rate $\cO(\frac{1}{(1-\delta)\rho^2}\log\frac{1}{\epsilon})$ in \citet{koloskova2019decentralized}, and recovers the best rate of \cref{alg_con3} in terms of $\rho$\citet{xiao2004fast}. Moreover, \cref{theo2} shows the linear convergence of both innovations and consensus errors to zero, which is consistent with our observations.
\end{remark}

\subsection{Development of COLD}

An extension of \cref{alg_con1} for the distributed optimization problem \cref{original} is proposed in \citet{koloskova2019decentralized} by directly adding a gradient step in the update of $X^{k+1}$.
The resulting algorithm is a quantized version of DGD and can only converge sublinearly.

To achieve a faster convergence rate, we are motivated by the state-of-the-art NIDS \cite{li2019decentralized} that has a linear convergence rate $\cO(\max\{\frac{L}{\mu},\frac{1}{\rho}\}\log\frac{1}{\epsilon})$ to find an $\epsilon$-optimal solution of \cref{original}, i.e., $\|\vx_i-\vx^\star\|^2\leq\epsilon,\forall i\in\cV$ \cite{xu2020distributed}. NIDS can be written compactly as follows 
\bee\label{nids}
X^{k+1}=\tilW(2X^k-X^{k-1}-\gamma\nabla F(X^k)+\gamma\nabla F(X^{k-1})),
\ene
where $\tilW=\frac{1}{2}(I+W)$ and the first iteration is initialized as $X^1=X^0-\gamma\nabla F(X^0)$. Clearly, it requires transmitting the exact $X^k$ with infinite precision. Similar to the case of the consensus algorithm in \cref{alg_con3}, a straightforward compression of \cref{nids} by  transmitting $\sQ(X^k)$ instead of $X^k$ fails to ensure exact convergence.

We adopt the idea of compressing the innovation vector to design a communication-efficient variant of \cref{nids}. Let $Y^k=2X^k-X^{k-1}-\gamma\nabla F(X^k)+\gamma\nabla F(X^{k-1})$. Then, \cref{nids} can be rewritten as $X^{k+1}=\tilW Y^k$. This is different from \cref{alg_con3} in that each node $i$ updates its state as a weighted average of $\vy_j^k$ rather than $\vx_j^k$ for all $j\in\cN_i$. We can show that $Y^k$ converges to $X^k$ by checking the fixed points of \cref{nids}. Thus, it is reasonable to view $Y^k$ as an approximation of $X^k$. By replacing $X^k$ in \cref{alg_con1} with $Y^k$ and introducing one more stepsize $\tau$, we obtain  COLD in the following form:
\bea\label{alg_opt11}
Y^{k}&=2X^k-X^{k-1}-\gamma\nabla F(X^k)+\gamma\nabla F(X^{k-1}),\\
\widehat Y^{k+1}&=\widehat Y^k +\sQ(Y^k-\widehat Y^k),\\
X^{k+1}&=Y^k+\tau\gamma(W-I)\widehat Y^{k+1},
\ena
where $\widehat Y^k$ is an auxiliary vector to track $Y^k$. The iteration is initialized by $X^{1}=Y^0=X^0-\gamma \nabla F(X^0)$ and $\widehat Y^{1}=\sQ(Y^0)$. Note that the introduction of the tunable stepsize $\tau$ is critical in both theoretical analysis and practical performance. Intuitively, $\gamma$ acts like the stepsize in the standard gradient method and $\tau$ plays a similar role of the stepsize in \cref{alg_con1}.

By introducing another auxiliary vector $\psi_i^k\in\bR^d$ for each node and defining $\Psi^k=[\psi_1^k,\cdots,\psi_n^k]^\T\in\bR^{n\times p}$, COLD in \cref{alg_opt11} is equivalent to the following form
\begin{subequations}\label{alg_opt1}
	\begin{align}
		Y^k          & =X^k-\gamma \nabla F(X^k)-\gamma \Psi^k,\label{alg_opt1a}   \\
		\widehat Y^{k+1} & =\widehat Y^k +\sQ(Y^k-\widehat Y^k),\label{alg_opt1b}              \\
		\Psi^{k+1}   & =\Psi^k+\tau(I-W)\widehat Y^{k+1}\label{alg_opt1c},             \\
		X^{k+1}      & =X^k-\gamma\nabla F(X^k)-\gamma\Psi^{k+1},\label{alg_opt1d}
	\end{align}
\end{subequations}
and $\Psi^{1}=\Psi^{0}=0$ for $k=0$. The equivalence can be readily checked by eliminating $\Psi^k$ in \cref{alg_opt1}. COLD of this form is more desirable in implementation since it does not require node $i$ to re-compute or store $\nabla f_i(\vx_i^{k-1})$. The implementation details are summarized in \cref{cold}, where each node only transmits the compressed vector $\sQ(\vy_i^k-\widehat \vy_i^k)$ to its neighbors.
\begin{algorithm}[!t]
	\caption{The COLD --- from the view of node $i$}\label{cold}
	\begin{algorithmic}[1]
		\REQUIRE The initial point $\vx_i^0$. Set $\vx_i^{1}=\vx_i^0-\gamma \nabla f_i(\vx_i^0)$ and $\psi_i^1=\widehat \vy_i^{1}=\widetilde\vy_i^1=\bzero$.
		\FOR {$k=1,2,\cdots$}
		\STATE Compute $\vy_i^{k}=\vx_i^k-\gamma\nabla f_i(\vx_i^k)-\gamma\psi_i^k$, $\vq_i^k=\sQ(\vy_i^k-\widehat \vy_i^k)$ and $\widehat\vy_i^{k+1}=\widehat\vy_i^{k}+\vq_i^k$.
		\STATE Send $\vq_i^k$ to all neighbors and receive $\vq_j^k$ from each neighbor $j\in\cN_i$.
		\STATE Update
		\bea
		\widetilde\vy_i^{k+1}&=\widetilde\vy_i^{k}+\tau\left(\vq_i^k-\sum\nolimits_{j\in\cN_i}[W]_{ij}\vq_j^k\right)\\
		\psi_i^{k+1}&=\psi_i^k+\widetilde\vy_i^{k+1}\\
		\vx_i^{k+1}&=\vx_i^k-\gamma\nabla f_i(\vx_i^k)-\gamma\psi_i^{k+1}.
		\ena
		\ENDFOR
	\end{algorithmic}
\end{algorithm}

\subsection{Convergence result}

The following lemma shows the equivalence between fixed points of \cref{alg_opt1} and optimal solutions of \cref{original2}.

\begin{lemma}\label{fixed_point}
	\begin{enumerate}[label=(\alph*),noitemsep]
		\item $(X^\star,\Psi^\star,\widehat Y^\star)$ is a fixed-point of \cref{alg_opt1} if and only if $(I-W)X^\star=0, \widehat Y^\star=X^\star, \Psi^\star=-\nabla F(X^\star)$.
		\item $X^\star$ is an optimal solution of \cref{original2} if and only if $(I-W)X^\star=0$ and $\nabla F(X^\star)\in\text{range}(I-W)$.
	\end{enumerate}
\end{lemma}
\begin{proof} (a) It can be easily checked that $(X^\star,\Psi^\star,\widehat Y^\star)$ satisfying the condition is a fixed-point. Conversely, it follows from \cref{alg_opt1d} that $\Psi^\star=-\nabla F(X^\star)$, and hence $Y^\star=\widehat Y^\star=X^\star$ by \cref{alg_opt1a,alg_opt1b}. Then, we obtain $(I-W)X^\star=0$ from \cref{alg_opt1c}.  (b) It directly follows from  the Karush-Kuhn-Tucker (KKT) theorem  for \cref{original2} \cite{boyd2004convex}.
\end{proof}

Since $\Psi^1=0$, it follows from \cref{alg_opt1c} that $\Psi^k\in\text{range}(I-W)$ which implies that $\Psi^\star\in\text{range}(I-W)$. Thus, \cref{fixed_point} allows us to focus only on the convergence of \cref{alg_opt1} to its fixed points. Define $\Theta=\tau^{-1}(I-W)^\dagger-\gamma I$, which is nonnegative-definite in the sequel. We first provide the convergence result of COLD for unbiased compressors under \cref{assum1}.
\begin{theo}\label{theo1}
	Suppose \cref{assum1,assum2,assum4} hold, and $\sQ(\cdot)$ is unbiased.   Let $\gamma\in\big(0,\frac{1}{\mu+L}\big)$, $\tau\in\big(0,\frac{(1-\delta)^2}{2\gamma(24\delta+(1-\delta)^2)}\big)$ in  \cref{alg_opt1} and $\ve^k=\bE\big(\|X^{k}-X^\star\|_{\gamma^{-1}}^2+\|\Psi^{k}-\Psi^\star\|_{\Theta+\gamma I}^2+\|\Psi^k-\Psi^{k-1}\|_{\Theta}^2+\frac{2(1+\delta)\tau}{1-\delta}\|\widehat Y^{k}-Y^{k-1}\|_{I-W}^2\big)$ where $\{X^k\}$, $\{\Psi^k\}$ and $\{\widehat Y^k\}$ are generated by \cref{alg_opt1}. Then, we have
	\bee
	\ve^{k+1}\leq \sigma \ve^{k},
	\ene
	where $\sigma=\max\big\{1-\frac{2\mu L\gamma}{\mu+L},1-\frac{\gamma\tau\rho}{2}\big\}<1$. In particular, if we set $\gamma=\frac{1}{2L}$ and $\tau=\frac{(1-\delta)^2}{2\gamma(24\delta+1)}$, then
	\bee
	\sigma= \max\Big\{\frac{L}{\mu+L},1-\frac{(1-\delta)^2\rho}{4(1+24\delta)}\Big\} .
	\ene
\end{theo}

\cref{theo1} shows that COLD converges at a linear rate $\cO\big(\max\{\frac{L}{\mu},\frac{1}{(1-\delta)^2\rho}\}\log\frac{1}{\epsilon}\big)$.  In comparison with the (centralized) gradient method of the convergence rate $\cO\big(\frac{L}{\mu}\log\frac{1}{\epsilon}\big)$ \cite{nesterov2013introductory},  we can explicitly quantify how the network and compression affect the rates in terms of two simple quantities, i.e,  $\delta$ and $\rho$.  

Since the uncompressed NIDS converges at a linear rate $\cO\big(\max\{\frac{L}{\mu},\frac{1}{\rho}\}\log\frac{1}{\epsilon}\big)$ \cite{li2019decentralized}, the $\delta$-contracted compression only affects the term related to the network.  Particularly, COLD converges at the same rate as NIDS over a network with a ``modified" adjacency matrix $W'$ such that $\lambda_2(W')=1-(1-\delta)^2\lambda_2(W)$, and reduces to NIDS if $\delta=0$. Note that COLD is  communication-efficient since each node only sends the compressed vector $\vq_i^k$.

\cref{theo1} also shows how to design a compressor to minimize the overall communication costs for COLD. Consider the unbiased compressor in \cref{exap_compressor}\ref{compressor2} with $p=\infty$. The total number of transmitted bits of node $i$ to obtain an $\epsilon$-optimal solution is $((l+1)d+b)\cO\big(\max\{\frac{L}{\mu},(1-\frac{1}{2^{l-1}})^{-2}\frac{1}{\rho}\}\log\frac{1}{\epsilon}\big)$, where $l+1$ is the encoding length. An optimal $l$ can be obtained by minimizing it, which depends on $\mu,L,\rho,d,b$ and is not larger than $\max\{0,\log_2{\frac{b}{d}}\}+5$. Thus, transmitting a minimum number of bits (e.g. 1-bit) per iteration may be not optimal for COLD in terms of the total communication load.

The proof of \cref{theo1} is based on the following two lemmas. \cref{lemma5} bounds the innovation and \cref{lemma9} bounds the distance of current decision variables to the fixed points, the proofs of which are relegated to Appendix \ref{adx2}. To facilitate presentation, we let $\widetilde X^k=X^k-X^\star$, $\widetilde\Psi^k=\Psi^k-\Psi^\star$ and abbreviate $\nabla F(X^k)$ and $\nabla F(X^\star)$ to $\nabla F^k$ and $\nabla F^\star$, respectively.
\begin{lemma}\label{lemma5}
	Under \cref{assum1,assum2,assum4}, it holds that 
		\bea\label{eq1_lemma5}
	&\bE\|\widehat Y^{k+1}-Y^k\|^2\leq\frac{2\delta}{1+\delta}\|\widehat Y^{k}-Y^{k-1}\|^2\\
	&+\frac{2\delta\gamma^2}{1-\delta}\|\Psi^k-\Psi^\star+\Psi^k-\Psi^{k-1}+\nabla F^k-\nabla F^\star\|^2.
	\ena
\end{lemma}

\begin{lemma}\label{lemma9}
	Under \cref{assum1,assum2,assum4}, it holds that
	\bea\label{eq1}
	&\|\widetilde X^{k+1}\|_{\gamma^{-1}}^2+\|\widetilde \Psi^{k+1}\|_{\Theta+\gamma I}^2+\|\Psi^{k+1}-\Psi^{k}\|_\Theta^2\\
	&\leq \Big(1-\frac{2\mu L\gamma}{\mu+L}\Big)\|\widetilde X^k\|_{\gamma^{-1}}^2+\|\widetilde \Psi^k\|_\Theta^2+2\langle\widetilde \Psi^{k+1},\widehat Y^{k+1}-Y^k\rangle\\
	&\quad+\Big(1-\frac{2\gamma^{-1}}{\mu+L}\Big)\|\nabla F^k-\nabla F^\star\|_{\gamma}^2.
	\ena
\end{lemma}

\begin{proof}[Proof of \cref{theo1}]
	Notice that
	\bea
	&2\langle\Psi^{k+1}-\Psi^\star,\widehat Y^{k+1}-Y^k\rangle\\
	&\overset{\cref{alg_opt1c}}{=}2\langle\Psi^{k}-\Psi^\star+\tau(I-W)Y^k,\widehat Y^{k+1}-Y^k\rangle\\
	&\quad+\|\widehat Y^{k+1}-Y^k\|_{2\tau(I-W)}^2\\
	&=2\langle\Psi^{k}-\Psi^\star+\tau(I-W)Y^k,\sQ(Y^k-\widehat Y^k)-(Y^k-\widehat Y^k)\rangle\\
	&\quad+\|\widehat Y^{k+1}-Y^k\|_{2\tau(I-W)}^2.
	\ena
	Since $\sQ$ is unbiased, we have
	\bee
	\bE[2\langle\Psi^{k+1}-\Psi^\star,\widehat Y^{k+1}-Y^k\rangle]=\bE\|\widehat Y^{k+1}-Y^k\|_{2\tau(I-W)}^2.
	\ene

	Let $\Lambda=\frac{4}{1-\delta}\tau(I-W)$ and add $\bE\|\widehat Y^{k+1}-Y^k\|_{\Lambda}^2$ to both sides of $\cref{eq1}$. We obtain that
	\begin{align}
	&\|\widetilde X^{k+1}\|_{\gamma^{-1}}^2+\|\widetilde \Psi^{k+1}\|_{\Theta+\gamma I}^2+\|\Psi^{k+1}-\Psi^{k}\|_\Theta^2\\
	&\quad+\bE\|\widehat Y^{k+1}-Y^k\|_{\Lambda-2\tau(I-W)}^2\\
	%%%%%%%%%%%%%%%%%%%%%%%%%%%%%%%%%%%%%%%%%%%
	&\overset{\cref{eq1_lemma5},\cref{eq1}}{\leq}\Big(1-\frac{2\mu L\gamma}{\mu+L}\Big)\|\widetilde X^k\|_{\gamma^{-1}}^2+\frac{2\delta}{1+\delta}\|\widehat Y^{k}-Y^{k-1}\|_{\Lambda}^2\\
	&\;+\Big(1-\frac{2\gamma^{-1}}{\mu+L}\Big)\|\nabla F^k-\nabla F^\star\|_{\gamma}^2+\|\widetilde \Psi^k\|_\Theta^2\\
	&\;+\frac{3\delta\gamma^2}{1-\delta}\Big((1+\eta)\|\widetilde \Psi^k\|_{\Lambda}^2+(1+\eta^{-1})\|\Psi^k-\Psi^{k-1}\|_{\Lambda}^2\Big)\\
	&\;+\frac{6\delta\gamma^2}{1-\delta}\|\nabla F^k-\nabla F^\star\|_{\Lambda}^2\\
	%%%%%%%%%%%%%%%%%%%%%%%%%%%%%%%%%%%%%%%%%%%%
	&\leq\Big(1-\frac{2\mu L\gamma}{\mu+L}\Big)\|\widetilde X^k\|_{\gamma^{-1}}^2+\frac{2\delta}{1+\delta}\|\widehat Y^{k}-Y^{k-1}\|_{\Lambda}^2\\
	&\;+\|\nabla F^k-\nabla F^\star\|_{\gamma I-\frac{2}{\mu+L}I+\frac{6\delta\gamma^2\Lambda}{1-\delta}}^2+\|\widetilde \Psi^k\|_\Theta^2\\
	&\;+\frac{\gamma}{2}\|\widetilde \Psi^k\|^2+\frac{3\delta\gamma^2}{1-\delta}(1+\eta^{-1})\|\Psi^k-\Psi^{k-1}\|_{\Lambda}^2\\
	%%%%%%%%%%%%%%%%%%%%%%%%%%%%%%%%%%%%%%%%%%%%
	&\leq\Big(1-\frac{2\mu L\gamma}{\mu+L}\Big)\|\widetilde X^k\|_{\gamma^{-1}}^2+\|\widehat Y^{k}-Y^{k-1}\|_{\frac{8\delta\tau(I-W)}{1-\delta^2}}^2\\
	&\quad+\|\widetilde \Psi^k\|_{\Theta+\frac{\gamma}{2} I}^2+\frac{\lambda_1(\Theta)+\gamma/2}{\lambda_1(\Theta)+\gamma}\|\Psi^k-\Psi^{k-1}\|_{\Theta}^2
	\end{align}
	where the second inequality is obtained by setting $\eta=\frac{(1-\delta)^2}{24\delta}\gamma^{-1}\tau^{-1}(1-\lambda_n(W))^{-1}-1> \frac{(1-\delta)^2}{24\delta}$. The last inequality follows from $\gamma I-\frac{2}{\mu+L}I+\frac{6\delta\gamma^2\Lambda}{1-\delta}=\frac{24\delta}{(1-\delta)^2}\gamma^2\tau(I-W)+\gamma-\frac{2}{\mu+L}I\preceq 2\gamma-\frac{2}{\mu+L}I\preceq 0$ and the following relation
	\bee
	\frac{3\delta\gamma^2(1+\eta^{-1})}{1-\delta}\Lambda\preccurlyeq \frac{\lambda_1(\Theta)+\gamma/2}{\lambda_1(\Theta)+\gamma}\Theta,
	\ene
	which follows from the following inequalities
	\bea
	&\frac{3\delta\gamma(1+\eta)\Lambda}{1-\delta}\preccurlyeq\frac{12\delta\gamma(1+\eta)\tau(1-\lambda_n(W))}{(1-\delta)^2}I= \frac{1}{2}I
	\ena
	and
	\bea
	\frac{1}{2}I&\preccurlyeq\frac{\xi-1/2}{\xi}\eta(\xi-1)I\preccurlyeq\frac{\lambda_1(\Theta)+\gamma/2}{\lambda_1(\Theta)+\gamma}\eta(\xi-1)I\\
	&\preccurlyeq\frac{\lambda_1(\Theta)+\gamma/2}{\lambda_1(\Theta)+\gamma}\eta\gamma^{-1} \Theta
	\ena
	where $\xi=\gamma^{-1}\tau^{-1}(1-\lambda_n(W))^{-1}\geq\frac{24\delta}{(1-\delta)^2}+1$.

	Hence, it holds that
	\bea
	&\|\widetilde X^{k+1}\|_{\gamma^{-1}}^2+\|\widetilde \Psi^{k+1}\|_{\Theta+\gamma I}^2+\|\Psi^{k+1}-\Psi^{k}\|_\Theta^2\\
	&\quad+\bE\|\widehat Y^{k+1}-Y^k\|_{\frac{2(1+\delta)}{1-\delta}\tau(I-W)}^2\\
	&\leq\Big(1-\frac{2\mu L\gamma}{\mu+L}\Big)\|\widetilde X^k\|_{\gamma^{-1}}^2+\frac{\lambda_1(\Theta)+\gamma/2}{\lambda_1(\Theta)+\gamma}\|\widetilde \Psi^k\|_{\Theta+\gamma I}^2\\
	&\quad+\frac{4\delta}{(1+\delta)^2}\|\widehat Y^{k}-Y^{k-1}\|_{\frac{2(1+\delta)}{1-\delta}\tau(I-W)}^2\\
	&\quad+\frac{\lambda_1(\Theta)+\gamma/2}{\lambda_1(\Theta)+\gamma}\|\Psi^k-\Psi^{k-1}\|_{\Theta}^2\\
	&\leq\sigma\bigg(\|\widetilde X^k\|_{\gamma^{-1}}^2+\|\widetilde \Psi^k\|_{\Theta+\gamma I}^2\\
	&\qquad+\|\Psi^k-\Psi^{k-1}\|_{\Theta}^2+\|\widehat Y^{k}-Y^{k-1}\|_{\frac{2(1+\delta)}{1-\delta}\tau(I-W)}^2\bigg)
	\ena
	where
	$
	\sigma=\max\Big\{1-\frac{2\mu L\gamma}{\mu+L},1-\frac{\gamma\tau(1-\lambda_2(W))}{2}\Big\}
	$
	and we used the relation
	\bea
	\frac{\lambda_1(\Theta)+\gamma/2}{\lambda_1(\Theta)+\gamma}&=\frac{\gamma^{-1}\tau^{-1}(1-\lambda_2(W))^{-1}-1/2}{\gamma^{-1}\tau^{-1}(1-\lambda_2(W))^{-1}}\\
	&=1- \frac{\gamma\tau(1-\lambda_2(W))}{2}\geq \frac{4\delta}{(1+\delta)^2}.
	\ena
	The proof is completed by taking expectations on both sides.
\end{proof}

Now, we turn to general biased compressors. Different from the unbiased case, the convergence requires a small $\delta$ to compensate the biasedness caused by the compressor.

\begin{theo}\label{theo4}
	Suppose that \cref{assum1,assum2,assum4} hold, and $\delta'\triangleq\frac{4\sqrt{\delta}}{1-\delta}<1$. Let $\{X^k\}$, $\{\Psi^k\}$ and $\{\widehat Y^k\}$ be generated by \cref{alg_opt1}, $\gamma\in(0,\frac{1-\delta}{\mu+L})$, $\tau\in(0,\frac{1-\delta}{2\gamma(1-\lambda_n(W))})$, and $\ve^k=\bE\big(\|X^{k}-X^\star\|_{\gamma^{-1}}^2+\|\Psi^{k}-\Psi^\star\|_{\Theta+(1-{\sqrt{\delta'}}/{2})\gamma I}^2+\|\Psi^k-\Psi^{k-1}\|_{\Theta}^2+{\frac{1+\delta-4\sqrt{\delta}}{2\gamma\sqrt{\delta\delta'}}}\|\widehat Y^{k}-Y^{k-1}\|\big)$. Then, it holds that
	\bee
	\ve^{k+1}\leq \sigma \ve^{k},
	\ene
	where $\sigma=\max\big\{1-\frac{2\mu L\gamma}{\mu+L},1-\gamma\tau\rho(1-\sqrt{\delta'})\big\}<1$.
\end{theo}

Compared to the contraction factor $\sigma=\max\big\{1-\frac{2\mu L\gamma}{\mu+L},1-\frac{\gamma\tau\rho}{2}\big\}$ for the unbiased compressors in \cref{theo1}, \cref{theo4} shows that the linear convergence rate for biased compressors has the same dependence on $\gamma,\tau,\mu,L$ and $\rho$ if $\delta$ is small, but reduces faster as $\delta$ increases. Fortunately, a moderate number of encoding bits often leads to a very small $\delta$ since it typically decreases exponentially in the encoding length as shown in \cref{exap_compressor}\ref{compressor2}. Moreover, the round-off error caused by finite-precision is a consequence of biased compression with a \emph{small} $\delta$. For instance, $\delta<1.2\times 10^{-7}$ when using 32-bit floating-point format for $|\vx|>2^{-126}$ \cite{higham2002accuracy,ieee754}.

\begin{remark}
	Linear convergence has been achieved in \citet{liu2020linear,kovalev2020linearly,li2021compressed}  for only unbiased $\delta$-contracted compressors. Though the time-varying uniform quantizers  in \cite{lee2018finite,magnusson2020maintaining,xiong2021quantized} are biased, it is unclear whether their algorithms can adapt to other compressors and how the network and compression affect the convergence rates. To the best of our knowledge, COLD is the first distributed algorithm with linear convergence rate for biased $\delta$-contracted compressors.
\end{remark}

\begin{proof}[Proof of \cref{theo4}]
	It follows from the Cauchy-Schwarz inequality \cite{boyd2004convex} that
	\bea
	&\bE[2\langle\Psi^{k+1}-\Psi^\star,\widehat Y^{k+1}-Y^k\rangle]\\
	&\leq\zeta\bE\|\widetilde \Psi^{k+1}\|^2+\zeta^{-1}\bE\|\widehat Y^{k+1}-Y^k\|^2,\forall \zeta>0.
	\ena
	Let $\eta=\frac{1-\delta}{1+\delta}$, $\zeta=\frac{\gamma}{1-\delta}\sqrt{2\delta(1+\delta)\delta'^{-1}(1+\eta)}=\frac{2\gamma}{1-\delta}\sqrt{\delta \delta'^{-1}}=\frac{\sqrt{\delta'}\gamma}{2}$ and $\Lambda=\frac{1+\delta}{\delta'(1-\delta)}\zeta^{-1}$. Adding $\bE\|\widehat Y^{k+1}-Y^k\|_{\Lambda}^2$ to both sides of $\cref{eq1}$, we obtain that
	\begin{align}\label{eq4_theo4}
		 & \|\widetilde X^{k+1}\|_{\gamma^{-1}}^2+\|\widetilde \Psi^{k+1}\|_{\Theta+\gamma I-\zeta I}^2+\|\Psi^{k+1}-\Psi^{k}\|_\Theta^2                                               \\
		 & \quad+\bE\|\widehat Y^{k+1}-Y^k\|_{\Lambda-\zeta^{-1}}^2                                                                                                                        \\
		 & \overset{\cref{eq1_lemma5},\cref{eq1}}{\leq}\Big(1-\frac{2\mu L\gamma}{\mu+L}\Big)\|\widetilde X^k\|_{\gamma^{-1}}^2+\|\widetilde \Psi^k\|_\Theta^2                         \\
		 & \quad+\Big(1-\frac{2\gamma^{-1}}{\mu+L}\Big)\|\nabla F^k-\nabla F^\star\|_{\gamma}^2                                                                                        \\
		 & \quad+\frac{2\delta}{1+\delta}\|\widehat Y^{k}-Y^{k-1}\|_{\Lambda}^2+\frac{2\delta\gamma^2(1+\eta)}{1-\delta}\|\widetilde \Psi^k\|_{\Lambda}^2                                  \\
		 & \quad+\frac{4\delta\gamma^2(1+\eta^{-1})}{1-\delta}\big(\|\Psi^k-\Psi^{k-1}\|_{\Lambda}^2+\|\nabla F^k-\nabla F^\star\|_{\Lambda}^2\big)                                    \\
		 & \leq\Big(1-\frac{2\mu L\gamma}{\mu+L}\Big)\|\widetilde X^k\|_{\gamma^{-1}}^2+\|\widetilde \Psi^k\|_\Theta^2                                                                 \\
		 & \quad+\Big(1-\frac{2\gamma^{-1}}{\mu+L}\Big)\|\nabla F^k-\nabla F^\star\|_{\gamma}^2                                                                                        \\
		 & \quad+\frac{2\delta\|\widehat Y^{k}-Y^{k-1}\|_{\Lambda-\zeta^{-1}}^2}{1-\delta'+(1+\delta')\delta}+\frac{4\delta\gamma^2\zeta^{-1}}{\delta'(1-\delta)^2}\|\widetilde \Psi^k\|^2 \\
		 & \quad+\frac{8\delta\gamma^2}{(1-\delta)^2}\big(\|\Psi^k-\Psi^{k-1}\|_{\Lambda}^2+\|\nabla F^k-\nabla F^\star\|_{\Lambda}^2\big)                                             \\
		 & \overset{(a)}{\leq}\Big(1-\frac{2\mu L\gamma}{\mu+L}\Big)\|\widetilde X^k\|_{\gamma^{-1}}^2+\|\widetilde \Psi^k\|_{\Theta+\zeta I}^2                                        \\
		 & \quad+\Big(1-\frac{2\gamma^{-1}}{\mu+L}\Big)\|\nabla F^k-\nabla F^\star\|_{\gamma}^2                                                                                        \\
		 & \quad+\frac{2\delta}{1-\delta'+(1+\delta')\delta}\|\widehat Y^{k}-Y^{k-1}\|_{\Lambda-\zeta^{-1}}^2                                                                              \\
		 & \quad+\frac{\sqrt{\delta'}\gamma(1+\delta)}{1-\delta}\big(\|\Psi^k-\Psi^{k-1}\|^2+\|\nabla F^k-\nabla F^\star\|^2\big)                                                      \\
		 & \leq\Big(1-\frac{2\mu L\gamma}{\mu+L}\Big)\|\widetilde X^k\|_{\gamma^{-1}}^2+\frac{\sqrt{\delta'}\gamma(1+\delta)}{1-\delta}\|\Psi^k-\Psi^{k-1}\|^2                         \\
		 & \quad+\big(1-\gamma\tau(1-\lambda_2(W))(1-\sqrt{\delta'})\big)\|\widetilde \Psi^k\|_{\Theta+\gamma I-\zeta I}^2                                                             \\
		 & \quad+\frac{2\delta}{1-\delta'+(1+\delta')\delta}\|\widehat Y^{k}-Y^{k-1}\|_{\Lambda-\zeta^{-1}}^2                                                                              
	\end{align}
	where we used $\frac{4\delta\gamma^2\zeta^{-1}}{\delta'(1-\delta)^2}=\zeta$ and $\frac{8\delta\gamma^2\Lambda}{(1-\delta)^2}=\frac{16\delta(1+\delta)\gamma}{\delta'\sqrt{\delta'}(1-\delta)^3}=\frac{\sqrt{\delta'}\gamma(1+\delta)}{1-\delta}$ to obtain $(a)$. The last inequality follows from $\gamma-\frac{2}{\mu+L}+\frac{\sqrt{\delta'}\gamma(1+\delta)}{1-\delta}\leq 0$ and
	\bea
	\Theta+\zeta I&\preceq \frac{\tau^{-1}(1-\lambda_2(W))^{-1}-\gamma+\zeta}{\tau^{-1}(1-\lambda_2(W))^{-1}-\zeta} (\Theta+\gamma I-\zeta I)\\
	&\preceq(1-\gamma\tau(1-\lambda_2(W))(1-\sqrt{\delta'}))(\Theta+\gamma I-\zeta I).
	\ena

	Since $\tau\leq\frac{1-\delta}{2\gamma(1-\lambda_n(W))}$, we have
	\bea
	\frac{\sqrt{\delta'}\gamma(1+\delta)}{1-\delta}I&\leq \frac{\sqrt{\delta'}(1+\delta)}{(1-\delta)(\gamma^{-1}\tau^{-1}(1-\lambda_n(W))^{-1}-1)} \Theta\\
	&\leq\sqrt{\delta'} \Theta.
	\ena
	Moreover, it is easy to show that $1-\gamma\tau(1-\lambda_2(W))(1-\sqrt{\delta'})\geq\sqrt{\delta'}\geq\frac{2\delta}{1-\delta'+(1+\delta')\delta}$. Then, it follows from \cref{eq4_theo4} that
	\bea
	&\|\widetilde X^{k+1}\|_{\gamma^{-1}}^2+\|\widetilde \Psi^{k+1}\|_{\Theta+\gamma I-\zeta I}^2\\
	&\quad+\|\Psi^{k+1}-\Psi^{k}\|_\Theta^2+\bE\|\widehat Y^{k+1}-Y^k\|_{\Lambda-\zeta^{-1}}^2\\
	&\leq\sigma\Big(\|\widetilde X^k\|_{\gamma^{-1}}^2+\|\widetilde \Psi^k\|_{\Theta+\gamma I-\zeta I}^2\\
	&\qquad+\|\Psi^{k}-\Psi^{k}\|_\Theta^2+\bE\|\widehat Y^{k}-Y^{k-1}\|_{\Lambda-\zeta^{-1}}^2\Big)
	\ena
	where
	$
	\sigma=\max\big(1-\frac{2\mu L\gamma}{\mu+L},1-\gamma\tau(1-\lambda_2(W))(1-\sqrt{\delta'})\big).
$

\end{proof}

\section{COLD with Dynamic Scaling}\label{sec4b}

Although COLD converges linearly for $\delta$-contracted compressors, it may diverge for the compressors under \cref{assum3}. To resolve it, we design a dynamic scaling mechanism in this section. A key observation is that the compression error for $\vx\in\{\vx:\|\vx\|_\p>1\}$ can be arbitrarily large under \cref{assum3}. Hence, we need to restrict the vector to be compressed to the unit ball via an appropriate scaling. Specifically, each node transmits $\sQ(\vx/s^k)$ instead of $\sQ(\vx)$ to neighbors, where $s^k>0$ is a time-varying scaling factor, i.e., the compressor is applied to a scaled version of $\vx$. The receiver then recovers $\vx$ by scaling $\sQ(\vx/s^k)$ back with $s^k$.
\subsection{Compressed consensus with dynamic scaling}
We first apply the idea on dynamic scaling mechanism to the distributed consensus problem. Define $\Delta^k = X^k-\widehat X^k$. The following algorithm is proposed by incorporating the dynamic scaling into \cref{alg_con1}, i.e.,
\bea\label{alg_con2}
\widehat X^{k+1}&=\widehat X^k+s^k \sQ(\Delta^k/s^k)\\
X^{k+1}&=X^k+\gamma(W-I)\widehat X^{k+1}
\ena
We show its linear convergence under \cref{assum3} and $s^k=c_s\beta^k$ for some $c_s>0,\beta\in(0,1)$ below. The proof is deferred to \cref{adx3}.
\begin{algorithm}[!t]
	\caption{The Dyna-COLD --- from the view of node $i$}\label{dyna-cold}
	\begin{algorithmic}[1]
		\REQUIRE The initial point $\vx_i^0$ and scaling factor sequence $\{s^k\}$. Set $\vx_i^{1}=\vx_i^0-\gamma \nabla f_i(\vx_i^0)$ and $\psi_i^1=\widehat \vy_i^{1}=\widetilde\vy_i^1=\bzero$.
		\FOR {$k=1,2,\cdots$}
		\STATE Compute $\vy_i^{k}=\vx_i^k-\gamma\nabla f_i(\vx_i^k)-\gamma\psi_i^k$, $\vq_i^k=\sQ((\vy_i^k-\widehat \vy_i^k)/s^k)$ and $\widehat\vy_i^{k+1}=\widehat\vy_i^{k}+s^k\cdot \vq_i^k$.
		\STATE Send $\vq_i^k$ to all neighbors and receive $\vq_j^k$ from each neighbor $j\in\cN_i$.
		\STATE Update
		\bea
		\widetilde\vy_i^{k+1}&=\widetilde\vy_i^{k}+\tau s^k\left(\vq_i^k-\sum\nolimits_{j\in\cN_i}[W]_{ij}\vq_j^k\right)\\
		\psi_i^{k+1}&=\psi_i^k+\widetilde\vy_i^{k+1}\\
		\vx_i^{k+1}&=\vx_i^k-\gamma\nabla f_i(\vx_i^k)-\gamma\psi_i^{k+1}.
		\ena
		\ENDFOR
	\end{algorithmic}
\end{algorithm}

\begin{theo}\label{theo5}
	Let \cref{assum2,assum3} hold, and $\{X^k\}$ and $\{\widehat X^k\}$ be generated by \cref{alg_con2}. Let $0< \gamma\leq \min\big\{\frac{(1-\delta)\rho}{2\delta(1+2c_{p})},1\big\}$ and $s^k=c_s \beta^k$ for any $\beta\in[\underline{\beta},1)$, where $\beta^k$ denotes the $k$-th power of $\beta$, $c_p=\sqrt{n}d^{|\frac{1}{2}-\frac{1}{p}|}$, $c_s\geq\|X^0-\bone\bar\vx^\T\|/\varsigma$, $\varsigma$ is given in \cref{eq4_theo5},  and $\underline\beta= \max\left\{1-\frac{\gamma\rho}{2}+\frac{2\delta\gamma^2c_p}{1-(1+2\gamma)\delta},\ \frac{1+\delta}{2}+\delta\gamma\big(1+\frac{2c_p}{\rho}\big)\right\}<1$. Then, it holds that
	\bea\label{eq1_theo5}
	\|X^k-\bone\bar\vx^\T\|&\leq \|X^0-\bone\bar\vx^\T\|\beta^k,\\
	\|\widehat X^k-X^k\|_{\text{max}}&\leq c_s\beta^k.
	\ena
\end{theo}

\cref{theo5} shows the linear convergence of \cref{alg_con2} for compressors satisfying \cref{assum3}, which includes the one-bit binary quantizer in \cref{exap_compressor}\ref{compressor1}. In contrast, \cref{alg_con1} fails to converge with this compressor as illustrated in \cref{sec5c}.

\begin{remark}
	While \cref{alg_con2} is motivated by \cref{alg_con1}, a special case of \cref{alg_con2} is studied in \cite{li2011distributed} with a different form for an element-wise uniform quantizer. In this view, \cref{theo5} generalizes their result to a broader class of compressors. Moreover, \cite{li2011distributed} only provides the asymptotic convergence rate where in Theorem 4 we establish a non-asymptotic linear convergence rate. 
\end{remark}
\vspace{-0.15cm}
\subsection{The Dyna-COLD}
Now, we  integrate the idea of dynamic scaling with COLD to design the Dyna-COLD as follows 
\begin{equation}\label{alg_opt2}
	\begin{aligned}
		Y^k          & =X^k-\gamma \nabla F(X^k)-\gamma \Psi^k,   \\
		\varDelta^k  & = Y^k -\widehat Y^k,                           \\
		\widehat Y^{k+1} & =\widehat Y^k +s^k\sQ(\varDelta^k/s^k),        \\
		\Psi^{k+1}   & =\Psi^k+\tau(I-W)\widehat Y^{k+1},             \\
		X^{k+1}      & =X^k-\gamma\nabla F(X^k)-\gamma\Psi^{k+1}.
	\end{aligned}
\end{equation}
  Different from \cref{alg_opt1}, each node of Dyna-COLD sends  $\sQ(\varDelta^k/s^k)$ rather than $\sQ(\varDelta^k)$ to its neighbors, and each neighbor scales it back by $s^k$ to obtain an approximation of $\varDelta^k$.
The implementation is very similar to \cref{cold} except that the updates in $\vq_i^k$ and $\widetilde\vy_i^{k+1}$. See Algorithm \ref{dyna-cold} for details. The linear convergence rate of Dyna-COLD is given below.
\begin{theo}\label{theo3}
	Under Assumptions \ref{assum3}-\ref{assum2}, and $\delta'\triangleq\frac{64\delta(1+\delta) c_{p}}{(1-\delta)^3\rho}<1$, where $c_p=\sqrt{n}d^{|\frac{1}{2}-\frac{1}{p}|}$. Let $\{X^k\}$, $\{\Psi^k\}$ and $\{\widehat Y^k\}$ be generated by \cref{alg_opt2}, $\tau\leq\frac{2\mu L}{\rho(\mu+L)},\
		\gamma\leq\min\big\{\frac{2}{\mu+L},\frac{1}{2\tau(1-\lambda_n(W))}, \frac{\tau(1-\lambda_n(W))}{L^2}\big\}$
	and $s^k=c\beta^k$, where
	\begin{itemize}[leftmargin=*,topsep=0pt,noitemsep,nolistsep]
		\item $\beta=\frac{1}{2}+\max\{\frac{2\nu}{4-\delta'\widetilde\tau\rho},\ \frac{\delta}{1+\delta}+\frac{16\delta\widetilde\tau c_{p}}{(1-\delta)^2(1-\nu)}\}<1$,
		      and $\widetilde\tau=\gamma\tau(1-\lambda_n(W))<\frac{1}{2}$,
		      $\nu=\max\big\{1-\frac{2\mu L\gamma}{\mu+L},1-\frac{1}{2}\gamma\tau\rho\big\}<1$;
		\item $c=\varsigma \widetilde c$, $\widetilde c\geq\max\big\{(\|X^{0}-X^\star\|_{\gamma^{-1}}^2+\|\Psi^{0}-\Psi^\star\|_{\Theta+\frac{\gamma}{2} I}^2)/\varsigma,\|\widehat Y^{0}-Y^{0}\|_{\text{max}}^2\big\}$, $\varsigma=\frac{1}{2}(\varsigma_1+\varsigma_2)$, and
		      $
			      \varsigma_1=\frac{8\bar{c}\gamma \widetilde\tau\nu(1+\delta)}{(1-\delta)^3(1-\widetilde\tau)-16\delta(1+\delta)\widetilde\tau \underline{c}\underline{c}},\ \varsigma_2=\frac{(1-\nu)\gamma}{2 \underline{c}\delta}.
		      $
	\end{itemize}

	Let $\ve_1^k=\|X^{k}-X^\star\|_{\gamma^{-1}}^2+\|\Psi^{k}-\Psi^\star\|_{\Theta+\frac{\gamma}{2} I}^2$ and $\ve_2^k=\|\widehat Y^{k}-Y^{k}\|_{\text{max}}^2$. Then, it holds that
	\bee
	\ve_1^{k}\leq \widetilde c\beta^k,\ \ve_2^{k}\leq  c\beta^k,\ \forall k\geq 0.
	\ene
\end{theo}

The proof is given in Appendix \ref{adx4}. \cref{theo3} shows that Dyna-COLD achieves a linear convergence rate by setting $s^k=c\beta^k$ for a sufficiently large $c>0$ and for some $\beta\in(0,1)$. We use the mathematical induction  to show that the innovation  $\|\varDelta^k\|_\text{max}\leq s^k$. If the compression error $\delta$ is small, then $\beta$ can be set close to ${(\nu+1)}/{2}$ such that Dyna-COLD has a similar convergence rate with the NIDS.

We note that $s^k$ in Dyna-COLD can be \emph{arbitrarily} chosen without affecting the linear convergence if $\sQ$ further satisfies \cref{assum1}. To elucidate it, define $\widetilde\sQ(\vx)=s^k\sQ(\vx/s^k)$, and Dyna-COLD is identical to COLD with compressor $\widetilde\sQ(\vx)$. Then, $\|\widetilde\sQ(\vx)-\vx\|^2=(s^k)^2\|\sQ(\vx/s^k)-\vx/s^k\|^2 \leq (s^k)^2\delta\|\vx/s^k\|^2= \delta\|\vx\|^2$, which also satisfies \cref{assum1}. Hence, the linear convergence follows immediately from \cref{theo1}. This suggests that Dyna-COLD is preferable to COLD, which is confirmed in numerical experiments. Moreover, the experiments also illustrate the robustness of Dyna-COLD to the choice of $s^k$, which can ease the parameter tuning process in implementation.

It is worth noting that linear convergence has been observed in our experiments even if $\delta$ is larger than the threshold in \cref{theo3}, where we test the 1-bit binary compressor in \cref{exap_compressor}\ref{compressor1} with $\delta=0.5$. We do not find other methods that achieve linear convergence with this compressor.

\begin{remark}
	Ref. \cite{yi2014quantized} proposes an exactly convergent algorithm for a uniform quantizer without convergence rate analysis by directly combining \cref{alg_con2} with a gradient step in the update of $X^{k+1}$, which is improved in \cite{doan2020fast} with an explicit sublinear convergence rate. Then, linear convergence is achieved in \cite{lee2018finite,magnusson2020maintaining,xiong2021quantized} by designing specified time-varying uniform compressors. In contrast, Dyna-COLD converges linearly for a class of compressors.
\end{remark}

\section{Numerical Experiments}\label{sec5}

In this section, we use numerical experiments to (a) validate our theoretical results, (b) examine the effects of different compressors on the convergence rates, and (c) compare COLD and Dyna-COLD with existing methods. 

\subsection{Setup}

{\bf Network.} We consider $n=20$ computing nodes connected as an Erd\H{o}s-R\'{e}nyi graph \cite{west1996introduction}, where any two nodes are linked with probability $2\ln(n)/n$. Note that $\ln(n)/n$ is the lowest probability to ensure a connected graph. Then, we use the Metropolis rule \cite{xiao2004fast} to construct $W$ to satisfy \cref{assum2}. All results are repeated 10 times on different graphs, and we report their average performance.

{\bf Tasks.} We consider two problems. The first one is a logistic regression problem on the MNIST dataset \cite{lecun1998mnist} with the cost function $f(\vx_1,\cdots,\vx_{10})=-\frac{1}{m}\sum_{i=1}^{m}\sum_{j=1}^{10}l_j^i\log\big(\frac{u_j}{\sum_{j'=1}^{10}u_{j'}}\big)+\frac{r}{2}\|\vx_j\|^2$, where $m=60000,l_j^i\in\{0,1\}$, $\vh^i\in\bR^{784}$, $u_j=\exp( \vx_j^\T \vh^i)$, and $r=0.1$. We sort the samples by their labels and then evenly divide them into $n$ parts to create heterogeneous local datasets, and each node has exclusive access to one of them. This setting is more difficult than the random partition \cite{koloskova2019decentralized}. The second problem is an average consensus problem where nodes are supposed to find the mean $\bar\vx=\sum_{i=1}^n{\vx}_i$, and ${\vx}_i\in\bR^{10000}$ are randomly generated from a normal distribution.

{\bf Compressors.} We test four compressors as shown in \cref{compressors-table}, including both unbiased and biased compressors satisfying Assumptions \ref{assum1} or \ref{assum3}. The required numbers of bits to transmit their outputs are given. In comparison, nodes in NIDS send $32d$ bits per iteration (32-bit floating point format).

\renewcommand{\tabularxcolumn}[1]{m{#1}}
\begin{table}[t]
	\caption{Tested compressors. `Bits' means the number of bits needed to encode the output of the compressor for input $\vx\in\bR^d$.}
	\label{compressors-table}
	% \vskip 0.15in
	\begin{center}
		\begin{small}
			% \begin{sc}
			\begin{tabularx}{\linewidth}{cXc}
				\toprule
				Name & Description                                                                                                & Bits    \\
				\midrule
				C1   & Unbiased stochastic quantizer  in \cref{exap_compressor}\ref{compressor2} with $l=2$ and $p=\infty$        & $3d+32$ \\
				\addlinespace[3pt]
				% \midrule
				C2   & Biased quantizer in \cref{exap_compressor}\ref{compressor3}  with $l=2$, $p=\infty$ and $\xi=0.5\bone_d$   & $3d+32$ \\
				\addlinespace[3pt]
				% \midrule
				C3   & Logarithmic quantizer in \cref{exap_compressor}\ref{compressor1} with $\cQ=\{\pm 2^{i}|i=-3,-2,\cdots,3\}$ & $4d$    \\
				\addlinespace[3pt]
				% \midrule
				C4   & 1-bit binary quantizer in \cref{exap_compressor}\ref{compressor1}                                          & $d$     \\
				\bottomrule
			\end{tabularx}
			% \end{sc}
		\end{small}
	\end{center}
	\vskip -0.1in
\end{table}

{\bf Algorithms.} We compare COLD and Dyna-COLD with NIDS \cite{li2019decentralized}, LEAD \cite{liu2020linear}, DQOA \cite{kovalev2020linearly}, and CHOCO-SGD \cite{koloskova2019decentralized}. NIDS involves no compression and serves as a baseline. In all experiments, we fix $s^k=3\|X^1\|_\text{max}\cdot 0.99^k$ in Dyna-COLD. We tune all hyperparameters of these algorithms from the grid $[0.01, 0.05, 0.1, 0.3, 0.5, 0.7, 1, 1.5]$ (1 hyperparameter for CHOCO-SGD; 2 for COLD, Dyna-COLD, and NIDS; 3 for LEAD and DQOA).

\subsection{Logistic regression}

The convergence rates for different algorithm-compressor combinations are shown in \cref{fig2,fig3}. We highlight the following observations:
\begin{itemize}[leftmargin=*,topsep=0pt,noitemsep,nolistsep]
	\item For C1 and C2 satisfying \cref{assum1}, COLD, Dyna-COLD, and LEAD have almost indistinguishable performance from the uncompressed algorithm NIDS w.r.t. iterations (Figs. \subref*{fig2a} and \subref*{fig2b}). This result extends the convergence analysis of LEAD which is only established for unbiased compressors. Moreover, DQOA only converges with the unbiased C1, and CHOCO-SGD cannot converge to the exact solution for both compressors.
	\item For C3 and C4 that satisfy only \cref{assum3}, Dyna-COLD outperforms other algorithms and is the only convergent algorithm (Figs. \subref*{fig2c} and \subref*{fig2d}). It has an almost identical convergence rate with NIDS in both cases even though the key parameter $s^k$ is not fine-tuned. This result shows the robustness of Dyna-COLD to $s^k$.
	\item Dyna-COLD with the 1-bit compressor achieves the lowest communication cost among all algorithms (\cref{fig3}). Moreover, all methods outperform NIDS w.r.t. transmitted bits, which shows the effectiveness of the communication compression.
\end{itemize}

\begin{figure}[!t]
	\centering
	\subfloat[C1]{\label{fig2a}\includegraphics[width=0.499\linewidth]{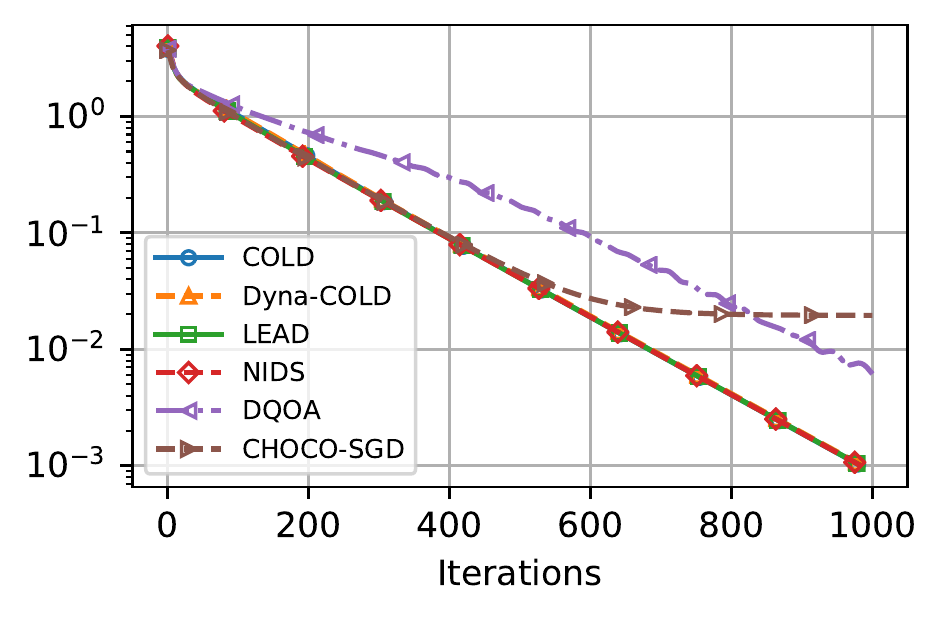}}
	\subfloat[C2]{\label{fig2b}\includegraphics[width=0.499\linewidth]{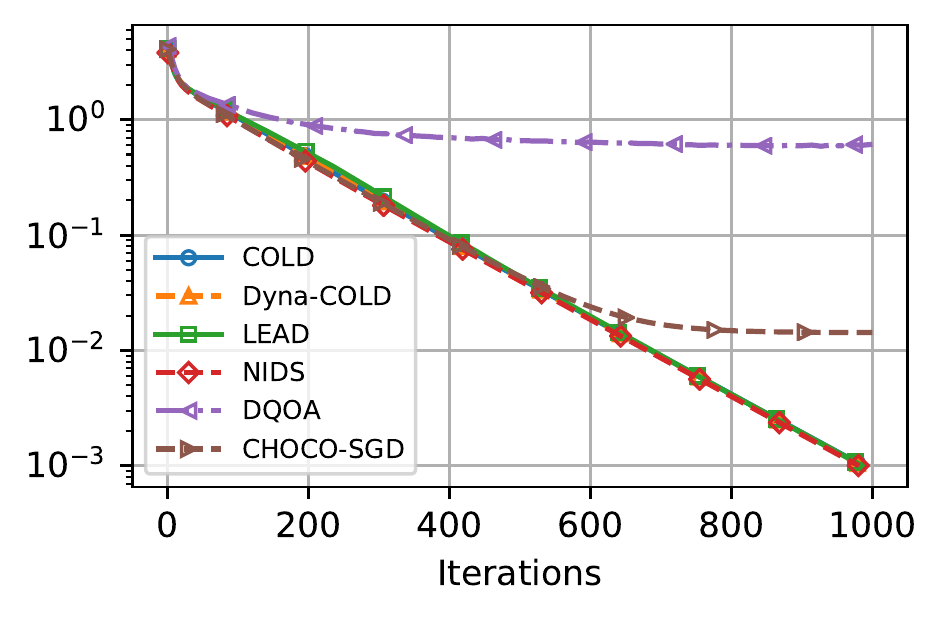}}\\
	\subfloat[C3]{\label{fig2c}\includegraphics[width=0.499\linewidth]{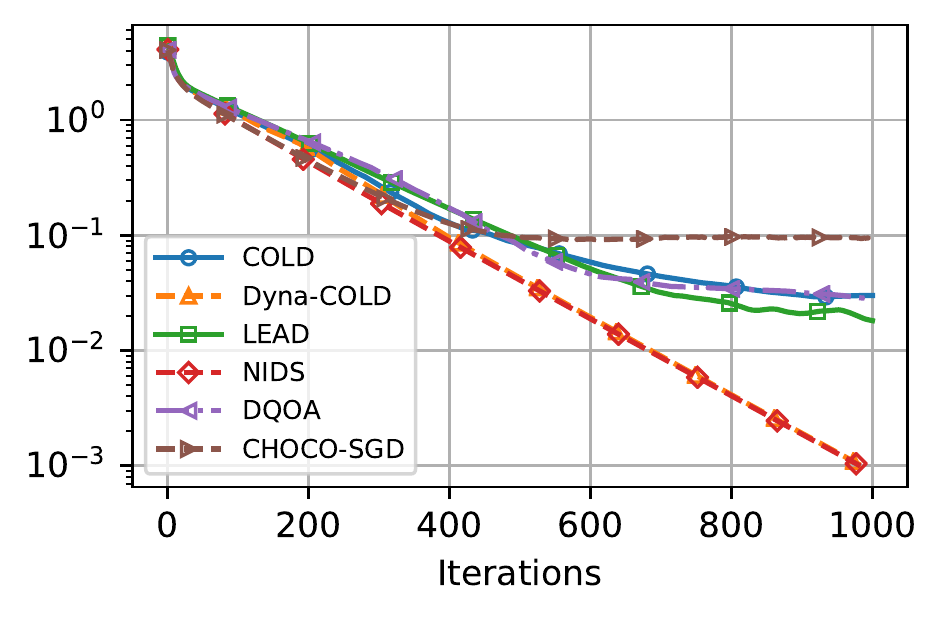}}
	\subfloat[C4]{\label{fig2d}\includegraphics[width=0.499\linewidth]{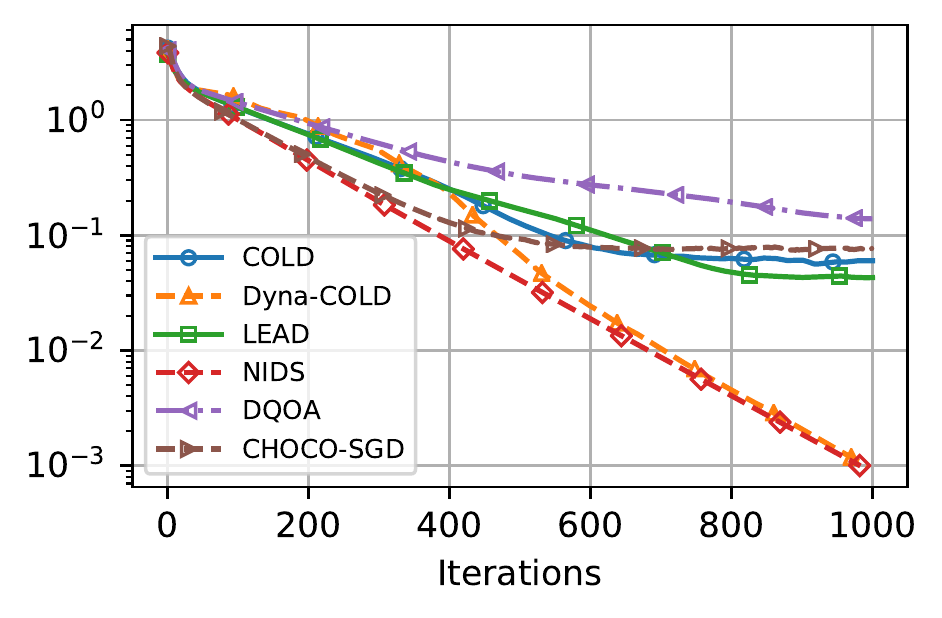}}
	\caption{Optimality gap ($\|\nabla f(\bar \vx^k)\|$) vs number of iterations for different algorithms with compressors C1-C4.}
	\label{fig2}
\end{figure}

\begin{figure}[!t]
	\centering
	\includegraphics[width=0.98\linewidth]{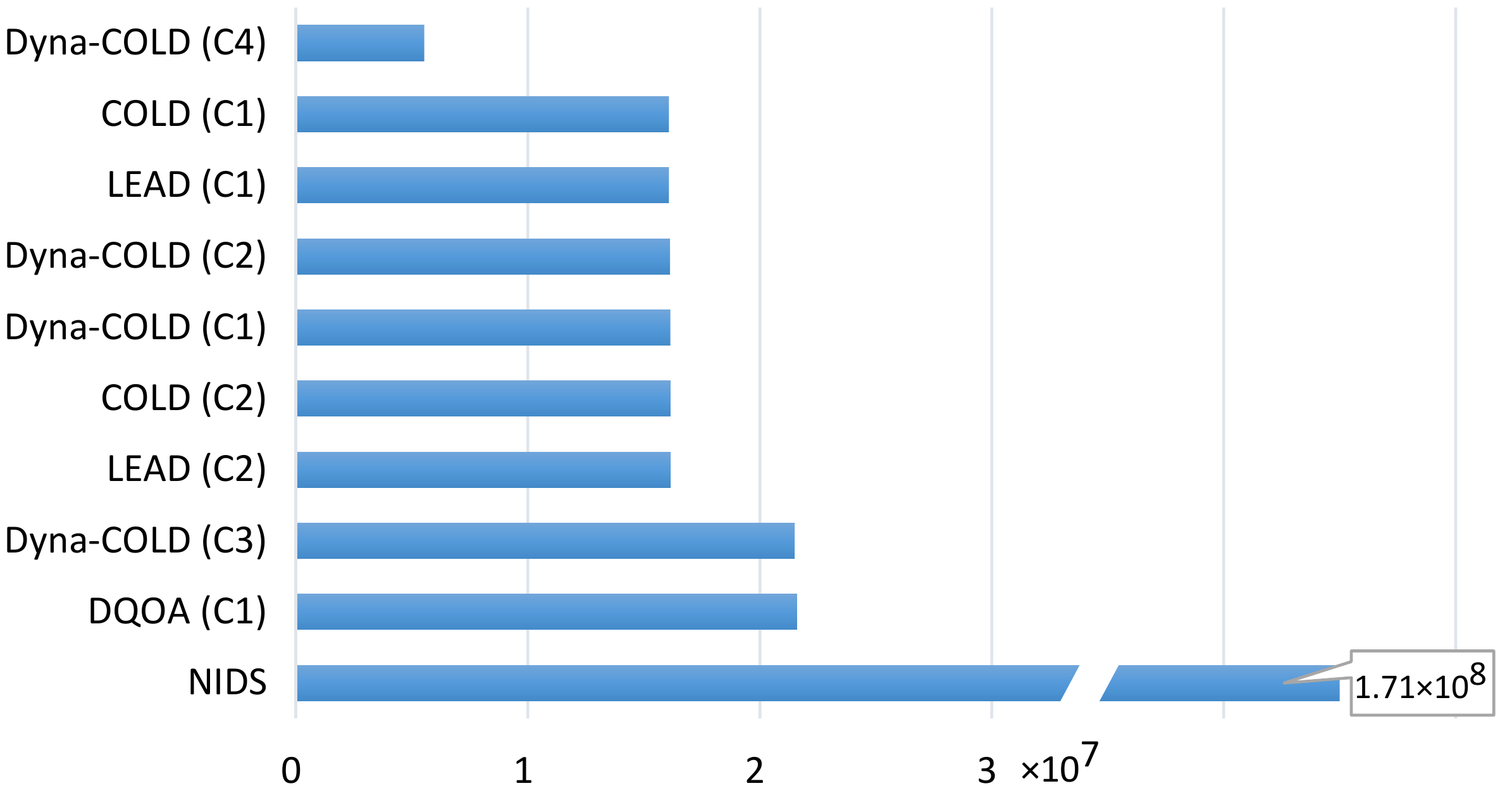} %fig4
	\caption{Transmitted bits for different algorithms to reach $\epsilon$-optimal solutions with $\epsilon=10^{-4}$. The algorithm-compressor combinations unable to find an $\epsilon$-optimal solution are not plotted.}
	\label{fig3}
\end{figure}

\subsection{Distributed consensus}\label{sec5c}

We compare the uncompressed method in \cref{alg_con3}, CHOCO-GOSSIP in \cref{alg_con1}, and our compressed consensus method with scaling (CCS) in \cref{alg_con2} on the consensus task. The result is depicted in \cref{fig1}. We have the following observations:
\begin{itemize}[leftmargin=*,topsep=0pt,noitemsep,nolistsep]
	\item CHOCO-GOSSIP and CCS can have similar convergence rates as \cref{alg_con3} w.r.t. iterations, which validates \cref{theo2}.
	\item CHOCO-GOSSIP cannot converge using C3 and C4, while CCS converges linearly.
	\item Biased compressors (e.g. C2) can converge faster than the unbiased one (C1) in terms of both iterations and transmitted bits. Moreover, transmitting the minimum number of bits per iteration (C4) does not lead to the minimum overall cost, which is consistent with our theoretical finding.
\end{itemize}

\begin{figure}[!t]
	\centering
	\includegraphics[width=\linewidth]{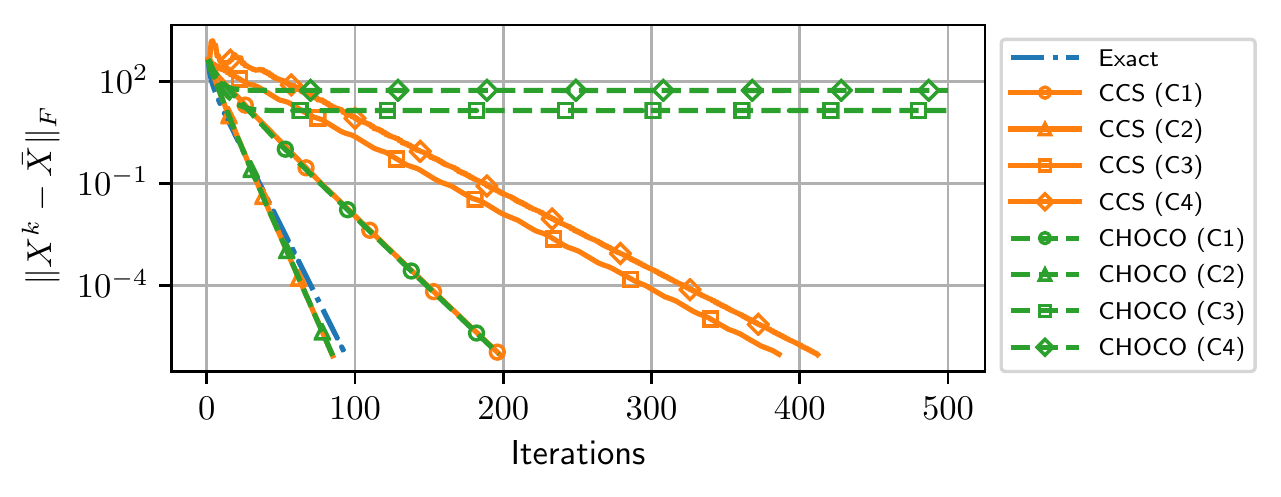}\\
	\includegraphics[width=\linewidth]{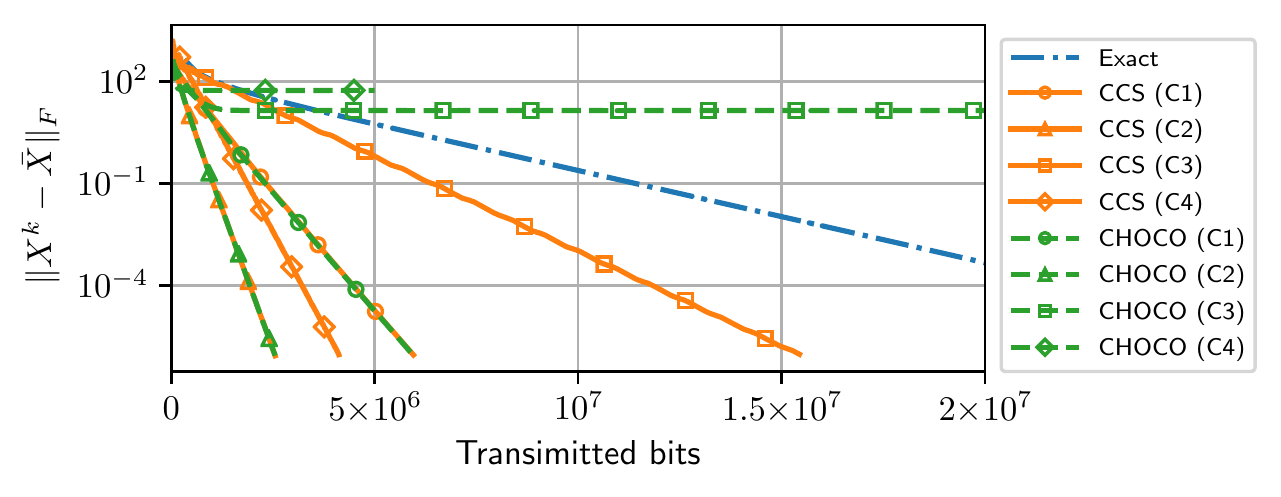}
	\caption{Consensus error vs iterations or transmitted bits. `Exact', `CHOCO', and `CCS' represent the algorithms in \cref{alg_con3}, \cref{alg_con1}, and \cref{alg_con2}, respectively.}
	\label{fig1}
\end{figure}

\section{Conclusion}\label{sec6}

We proposed two novel communication-efficient distributed algorithms based on innovation compression and dynamic scaling, which achieve linear convergence for two broad classes of compressors. Future works may focus on the extensions to the non-convex and stochastic settings.

\bibliographystyle{IEEEtran}
\bibliography{mybibf}

\appendices
\crefalias{section}{appendix}

\section{Proof of \cref{theo2}}\label{adx1}

The proof depends on several lemmas. \cref{lemma3} shows the contraction property of the adjacency matrix $W$.
\begin{lemma}\label{lemma3}
	For any $W$ satisfying \cref{assum2} and $\vx\in\text{range}(I-W)$, we have $\lambda_n(W)\|\vx\|^2\leq\|\vx\|_{W}^2\leq \lambda_2(W)\|\vx\|^2$.
\end{lemma}
\begin{proof}
	By \cref{assum2}, the eigenvalue 1 is the only eigenvalue of  $W$ with magnitude 1, and the corresponding eigenvector is $\bone_n$. Hence,  $\text{range}(I-W)=\text{range}(I-\frac{1}{n}\bone\bone^\T)$. Since $\vx\in\text{range}(I-W)$, there exists some $\vd$ such that $\vx=(I-\frac{1}{n}\bone\bone^\T)\vd$, which implies that $\|\vx\|_{W}^2=\vd^\T(I-\frac{1}{n}\bone\bone^\T)W(I-\frac{1}{n}\bone\bone^\T)\vd=\vd^\T(I-\frac{1}{n}\bone\bone^\T)(W-\frac{1}{n}\bone\bone^\T)(I-\frac{1}{n}\bone\bone^\T)\vd=\vx^\T(W-\frac{1}{n}\bone\bone^\T)\vx$. Note that $\lambda_n(W) I\preceq W-\frac{1}{n}\bone\bone^\T\preceq\lambda_2(W) I$, and thus the result follows.
\end{proof}

\begin{lemma}\label{lemma1}
	Under \cref{assum2} and let $\bar X=\bone\bar\vx^\T$. It holds that
	\bea\label{eq1_lemma1}
	&\|X^k-\bar X\|^2+\|\widehat X^{k}-X^{k-1}\|_{\gamma(W-I)}^2\\
	&=\|X^{k-1}-\bar X\|_{I+\gamma(W-I)}^2+\|\widehat X^{k}-\bar X\|_{\gamma(W-I)(I+\gamma(W-I))}^2.
	\ena
\end{lemma}
\begin{proof}
	It follows from \cref{alg_con1} and \cref{assum2} that
	\bea
	&X^k-\bar X=X^{k-1}-\bar X+\gamma(W-I)\widehat X^{k}\\
	&=X^{k-1}-\bar X+\gamma(W-I)(\widehat X^{k}-\bar X).
	\ena
	This implies that
	\bea\label{eq2_lemma1}
	\|X^k-\bar X\|^2&=\|X^{k-1}-\bar X+\gamma(W-I)(\widehat X^{k}-\bar X)\|^2\\
	&=\|X^{k-1}-\bar X\|^2+\|\widehat X^{k}-\bar X\|_{\gamma^2(W-I)^2}^2\\
	&\quad+2\langle X^{k-1}-\bar X,\widehat X^{k}-\bar X\rangle_{\gamma(W-I)}.
	\ena
	Since $\langle \va,\vb\rangle_M=\|\va+\vb\|_M^2-\|\va\|_M^2-\|\vb\|_M^2$ for any symmetric matrix $M$, we have
	\bea
	&2\langle X^{k-1}-\bar X,\widehat X^{k}-\bar X\rangle_{\gamma(W-I)}\\
	&=\|\widehat X^{k}-X^{k-1}\|_{\gamma(I-W)}^2-\|\widehat X^{k}-\bar X\|_{\gamma(I-W)}^2\\
	&\quad-\|X^{k-1}-\bar X\|_{\gamma(I-W)}^2
	\ena
	which jointly with \cref{eq2_lemma1} completes the proof.
\end{proof}

\begin{lemma}\label{lemma2}
	Under \cref{assum1,assum2}, it holds that
	\bea\label{eq1_lemma2}
	&\bE\|\widehat X^{k+1}- X^k\|^2\\
	&\leq\delta\Big(\|\widehat X^k-X^{k-1}\|_{I-\gamma(W-I)}^2+\|X^{k-1}-\bar X\|_{\gamma(W-I)}^2\\
	&\quad+\|\widehat X^{k}-\bar X\|_{\gamma(W-I)(\gamma(W-I)-I)}^2\Big).
	\ena
\end{lemma}
\begin{proof}
	We have from \cref{alg_con1} that
	\bea
	&\bE\|\widehat X^{k+1}- X^k\|^2=\bE\|\sQ(X^k-\widehat X^k)-(X^k-\widehat X^k)\|^2\\
	&\leq\delta\|\widehat X^k-X^k\|^2=\delta\|\widehat X^k-X^{k-1}-\gamma(W-I)(\widehat X^{k}-\bar X)\|^2\\
	&=\delta\big(\|\widehat X^k-X^{k-1}\|^2+2\langle X^{k-1}-\widehat X^k,\widehat X^{k}-\bar X\rangle_{\gamma(W-I)}\\
	&\quad+\|\widehat X^{k}-\bar X\|_{\gamma^2(W-I)^2}^2\big)\\
	&=\delta\big(\|\widehat X^k-X^{k-1}\|_{I-\gamma(W-I)}^2+\|X^{k-1}-\bar X\|_{\gamma(W-I)}^2\\
	&\quad+\|\widehat X^{k}-\bar X\|_{\gamma(W-I)(\gamma(W-I)-I)}^2\big)
	\ena
	where the inequality follows from \cref{assum1}, and the last equality follows from
	\bea
	&2\langle X^{k-1}-\widehat X^k,\widehat X^{k}-\bar X\rangle_{\gamma(W-I)}\\
	&=\|X^{k-1}-\bar X\|_{\gamma(W-I)}^2-\|X^{k-1}-\widehat X^k\|_{\gamma(W-I)}^2\\
	&\quad-\|\widehat X^{k}-\bar X\|_{\gamma(W-I)}^2.
	\ena
	Combining the above completes the proof.
\end{proof}

\begin{proof}[Proof of \cref{theo2}] Let $q=\frac{1+\gamma(1-\lambda_n(W))}{1-\gamma(1-\lambda_n(W))}\delta>\delta$. Multiplying \cref{eq1_lemma1} by $q$ and then adding it to \cref{eq1_lemma2}, we obtain
	\bea\label{eq1_theo1}
	&\|X^k-\bar X\|_{qI}^2+\bE\|\widehat X^{k+1}- X^k\|^2\\
	&\leq\|X^{k-1}-\bar X\|_{q(I+\gamma(W-I))+\delta\gamma(W-I)}^2\\
	&\quad+\|\widehat X^k-X^{k-1}\|_{\delta(I-\gamma(W-I))-q\gamma(W-I)}^2\\
	&\quad+\|\widehat X^{k}-\bar X\|_{\gamma(W-I)(qI-\delta I+(q+\delta)\gamma(W-I))}^2.
	\ena
	For any $\gamma\in(0,(1-\lambda_n(W))^{-1})$, it follows from \cref{assum2} that $\widetilde W=\sqrt{\gamma(I-W)}\succeq 0$. We have
	\bea
	&\gamma(W-I)(q I-\delta I+(q+\delta)\gamma(W-I))\\
	&=\frac{2\delta\gamma \widetilde W\big((I-W)-(1-\lambda_n(W))I\big)\widetilde W}{1-\gamma(1-\lambda_n(W))}\preceq  0.
	\ena
	It follows from \cref{eq1_theo1}  that
	\bea
	&\|X^k-\bar X\|_{qI}^2+\bE\|\widehat X^{k+1}- X^k\|^2\\
	&\leq\|X^{k-1}-\bar X\|_{q(I+\gamma(W-I))+\delta\gamma(W-I)}^2\\
	&\quad+\|\widehat X^k-X^{k-1}\|_{\delta(I-\gamma(W-I))-q\gamma(W-I)}^2\\
	&\leq\|X^{k-1}-\bar X\|_{qI+(q+\delta)\gamma(\lambda_2(W)-1)I}^2\\
	&\quad+\|\widehat X^k-X^{k-1}\|_{\delta I-(q+\delta)\gamma(\lambda_n(W)-1)I}^2\\
	&\leq\sigma\Big(\|X^{k-1}-\bar X\|_{qI}^2+\|\widehat X^k-X^{k-1}\|^2\Big)
	\ena
	where the second inequality follows from \cref{lemma3} by noticing that each column of $X^{k-1}-\bar X=(I-\frac{1}{n}\bone\bone^\T)X^{k-1}$ belongs to $\text{range}(I-W)$, and $\sigma$ is defined as follows
	\bea\label{eq3_theo1}
	\sigma&=\max\Big\{1-(1+{\delta}/{q})\gamma(1-\lambda_2), \delta+(q+\delta)\gamma(1-\lambda_n)\Big\}
	\\&=\max\Big\{1-\frac{2\gamma(1-\lambda_2)}{1+\gamma(1-\lambda_n)},\frac{\delta(1+\gamma(1-\lambda_n))}{1-\gamma(1-\lambda_n)}\Big\}\\
	&<1
	\ena
	where $\lambda_2$ and $\lambda_n$ are abbreviations for $\lambda_2(W)$ and $\lambda_n(W)$, respectively, and the last inequality follows from $\gamma<\frac{1-\delta}{(1+\delta)(1-\lambda_n(W))}$. Then, the first part of \cref{theo2} follows by taking full expectations on both sides.

	Substituting $\gamma=\frac{1-\delta}{(3+\delta)(1-\lambda_n(W))}$ into \cref{eq3_theo1}, we obtain
	\bea
	\sigma&=\max\Big\{1-\frac{(1-\delta)(1-\lambda_2(W))}{2(1-\lambda_n(W))},\frac{2\delta}{1+\delta}\Big\}\\
	&=1-\frac{(1-\delta)(1-\lambda_2(W))}{2(1-\lambda_n(W))}.
	\ena
	The desired result follows.
\end{proof}

\section{Proofs of \cref{lemma5,lemma9}}\label{adx2}

\subsection{Proof of \cref{lemma5}}

\begin{proof}
	We have
	\bea\label{eq2_lemma5}
	Y^k&\overset{\cref{alg_opt1a}}{=}X^k-\gamma \Psi^k-\gamma \nabla F^k\\
	&\overset{\cref{alg_opt1a}}{=}Y^{k-1}+X^k-X^{k-1}+\gamma (\Psi^{k-1}-\Psi^k)\\
	&\quad+\gamma (\nabla F^{k-1}-\nabla F^k)\\
	&\overset{\cref{alg_opt1d}}{=}Y^{k-1}+\gamma (\Psi^{k-1}-2\Psi^k)-\gamma\nabla F^k\\
	&\overset{\text{Lemma}~\ref{fixed_point}}{=}Y^{k-1}+\gamma (\Psi^{k-1}-\Psi^k)-\gamma(\Psi^k-\Psi^\star)\\
	&\qquad-\gamma(\nabla F^k-\nabla F^\star).
	\ena
	Therefore,
	\bea
	&\bE\|\widehat Y^{k+1}-Y^k\|^2\\
	&\overset{\cref{alg_opt1b}}{=}\bE\|\sQ(Y^k-\widehat Y^{k})-(Y^k-\widehat Y^{k})\|^2\leq\delta \|\widehat Y^{k}-Y^k\|^2\\
	&\overset{\cref{eq2_lemma5}}{=}\delta \|\widehat Y^{k}-Y^{k-1}+\gamma(\Psi^k-\Psi^\star)+\gamma (\Psi^k-\Psi^{k-1})\\
	&\quad+\gamma(\nabla F^k-\nabla F^\star)\|^2\\
	&\leq\frac{2\delta}{1+\delta}\|\widehat Y^{k}-Y^{k-1}\|^2\\
	&\quad+\frac{2\delta\gamma^2}{1-\delta}\|\Psi^k-\Psi^\star+\Psi^k-\Psi^{k-1}+\nabla F^k-\nabla F^\star\|^2
	\ena
	where the first inequality follows from \cref{assum1} and the last inequality used the relation $\|\va+\vb\|^2\leq(1+c)\|\va\|^2+(1+\frac{1}{c})\|\vb\|^2,\forall \va,\vb$ with $c=\frac{1-\delta}{1+\delta}$.
\end{proof}

\subsection{Proof of \cref{lemma9}}

We need several lemmas first.
\begin{lemma}\label{lemma6}
	Under \cref{assum1,assum2,assum4}. The following relation holds for all $k\geq 1$,
	\bea\label{eq3_lemma3}
	&\|X^{k+1}-X^\star\|_{\gamma^{-1}}^2+\|\Psi^{k+1}-\Psi^\star\|_\Theta^2\\
	&=\|X^{k}-X^\star\|_{\gamma^{-1}}^2+\|\Psi^{k}-\Psi^\star\|_\Theta^2-\|X^{k+1}-X^{k}\|_{\gamma^{-1}}^2\\
	&\quad-\|\Psi^{k+1}-\Psi^{k}\|_\Theta^2-2\langle X^k-X^\star,\nabla F^k-\nabla F^\star\rangle\\
	&\quad-2\langle \nabla F^k-\nabla F^\star, X^{k+1}-X^k\rangle\\
	&\quad+2\langle\Psi^{k+1}-\Psi^\star,\widehat Y^{k+1}-Y^k\rangle
	\ena
\end{lemma}
\begin{proof}
	We have
	\bea\label{eq1_lemma6}
	&\langle\widetilde \Psi^{k+1},\widetilde X^{k+1}\rangle\\
	&\overset{\cref{alg_opt1d}}{=}\langle\widetilde \Psi^{k+1},X^k-\gamma\nabla F^k-\gamma\Psi^{k+1}-X^\star\rangle\\
	&\overset{\cref{alg_opt1a}}{=}\langle\widetilde \Psi^{k+1}, Y^k+\gamma\Psi^{k}-\gamma\Psi^{k+1}-X^\star\rangle\\
	&=\langle\widetilde \Psi^{k+1}, (I-W)Y^k\rangle_{(I-W)^\dagger}\\
	&\quad+\langle\widetilde \Psi^{k+1},\gamma\Psi^{k}-\gamma\Psi^{k+1}-X^\star\rangle\\
	&\overset{\cref{alg_opt1c}}{=}\langle\widetilde \Psi^{k+1},\tau^{-1}(\Psi^{k+1}-\Psi^k)\\
	&\quad-(I-W)(\widehat Y^{k+1}-Y^k)\rangle_{(I-W)^\dagger}+\gamma\langle\widetilde \Psi^{k+1},\Psi^{k}-\Psi^{k+1}\rangle\\
	&=\langle\widetilde \Psi^{k+1},\Psi^{k+1}-\Psi^k\rangle_\Theta-\langle\widetilde \Psi^{k+1},\widehat Y^{k+1}-Y^k\rangle
	\ena
	where we used $\widetilde \Psi^k\in\text{range}(I-W)$ and hence $\langle\widetilde \Psi^{k+1},\vd\rangle_{(I-W)^\dagger(I-W)}=\langle\widetilde \Psi^{k+1},\vd\rangle,\forall \vd$ and $\langle\widetilde \Psi^{k+1},X^\star\rangle=0$. Therefore,
	\bea\label{eq2_lemma6}
	&\langle \widetilde X^k,\nabla F^k-\nabla F^\star\rangle\\
	&\overset{\cref{alg_opt1d}}{=}\langle \widetilde X^k,\gamma^{-1}(X^k-X^{k+1})-\Psi^{k+1}+\Psi^\star\rangle\\
	&=\langle \widetilde X^k,X^k-X^{k+1}\rangle_{\gamma^{-1}}-\langle \widetilde X^{k+1},\widetilde \Psi^{k+1}\rangle\\
	&\quad+\langle X^{k+1}-X^k,\widetilde \Psi^{k+1}\rangle\\
	&\overset{\cref{eq1_lemma6}}{=}\langle \widetilde \Psi^{k+1}-\gamma^{-1}(\widetilde X^k),X^{k+1}-X^k\rangle\\
	&\quad-\langle\widetilde \Psi^{k+1},\Psi^{k+1}-\Psi^k\rangle_\Theta+\langle\widetilde \Psi^{k+1},\widehat Y^{k+1}-Y^k\rangle\\
	&\overset{\cref{alg_opt1d}}{=}\langle \nabla F^k-\nabla F^\star+\gamma^{-1}(\widetilde X^{k+1}),X^k-X^{k+1}\rangle\\
	&\quad-\langle\widetilde \Psi^{k+1},\Psi^{k+1}-\Psi^k\rangle_\Theta+\langle\widetilde \Psi^{k+1},\widehat Y^{k+1}-Y^k\rangle\\
	&=\langle \widetilde X^{k+1}, X^k-X^{k+1}\rangle_{\gamma^{-1}}+\langle \nabla F^k-\nabla F^\star, X^k-X^{k+1}\rangle\\
	&\quad-\langle\widetilde \Psi^{k+1},\Psi^{k+1}-\Psi^k\rangle_\Theta+\langle\widetilde \Psi^{k+1},\widehat Y^{k+1}-Y^k\rangle
	\ena
	Using $2\langle \va,\vb  \rangle=\|\va+\vb\|^2-\|\va\|^2-\|\vb\|^2$, we have
	\bea
	&2\langle \widetilde X^k,\nabla F^k-\nabla F^\star\rangle+2\langle \nabla F^k-\nabla F^\star, X^{k+1}-X^k\rangle\\
	&\overset{\cref{eq2_lemma6}}{=}2\langle \widetilde X^{k+1}, X^k-X^{k+1}\rangle_{\gamma^{-1}}+2\langle\widetilde \Psi^{k+1},\Psi^k-\Psi^{k+1}\rangle_\Theta\\
	&\quad+2\langle\widetilde \Psi^{k+1},\widehat Y^{k+1}-Y^k\rangle\\
	&=\|\widetilde X^k\|_{\gamma^{-1}}^2-\|\widetilde X^{k+1}\|_{\gamma^{-1}}^2-\|X^{k+1}-X^{k}\|_{\gamma^{-1}}^2\\
	&\quad+\|\widetilde \Psi^k\|_\Theta^2-\|\widetilde \Psi^{k+1}\|_\Theta^2-\|\Psi^{k+1}-\Psi^{k}\|_\Theta^2\\
	&\quad+2\langle\widetilde \Psi^{k+1},\widehat Y^{k+1}-Y^k\rangle
	\ena
	The desired result is obtained.
\end{proof}

\begin{lemma}\label{lemma8}
	Under \cref{assum4}, we have for all $X,Y\in\bR^{n\times p}$ that
	\bea
	&\langle X-Y,\nabla F(X)-\nabla F(Y)\rangle\\
	&\geq \frac{\mu L}{\mu+L}\|X-Y\|^2+\frac{1}{\mu+L}\|\nabla F(X)-\nabla F(Y)\|^2.
	\ena
\end{lemma}
\begin{proof}
	It follows from Theorem 2.1.12 in \cite{nesterov2013introductory} that
	\bea
	&\langle\nabla f_i(\vx)-\nabla f_i(\vy),\vx-\vy\rangle\\
	&\geq\frac{\mu L}{\mu+L}\|\vx-\vy\|^2+\frac{1}{\mu+L}\|\nabla f_i(\vx)-\nabla f_i(\vy)\|^2,\forall i\in\cV.
	\ena
	The result then follows directly from the definition of $\nabla F$.
\end{proof}

\begin{proof}[Proof of \cref{lemma9}]
	For any symmetric $\Theta$, we have $2\langle \va,\vb  \rangle_\Theta=\|\va+\vb\|_{\Theta}^2-\|\va\|_{\Theta}^2-\|\vb\|_{\Theta}^2,\forall \va,\vb$, and hence
	\bea\label{eq1_lemma9}
	&2\langle \widetilde X^k,\nabla F^k-\nabla F^\star\rangle+2\langle \nabla F^k-\nabla F^\star, X^{k+1}-X^k\rangle\\
	&=2\langle \widetilde X^k,\nabla F^k-\nabla F^\star\rangle\\
	&\quad+\|X^{k+1}-X^k+\gamma\nabla F^k-\gamma\nabla F^\star\|_{\gamma^{-1}}^2\\
	&\quad-\|\nabla F^k-\nabla F^\star\|_{\gamma}^2-\|X^{k+1}-X^k\|_{\gamma^{-1}}^2\\
	&\geq \frac{2\mu L\gamma}{\mu+L}\|\widetilde X^k\|_{\gamma^{-1}}^2-\|X^{k+1}-X^k\|_{\gamma^{-1}}^2+\|\widetilde \Psi^{k+1}\|_{\gamma}^2\\
	&\quad+\Big(\frac{2\gamma^{-1}}{\mu+L}-I\Big)\|\nabla F^k-\nabla F^\star\|_{\gamma}^2
	\ena
	where the inequality follows from \cref{lemma8} and \cref{alg_opt1c}. Therefore,
	\bea
	&\|\widetilde X^{k+1}\|_{\gamma^{-1}}^2+\|\widetilde \Psi^{k+1}\|_\Theta^2-2\langle\widetilde \Psi^{k+1},\widehat Y^{k+1}-Y^k\rangle\\
	&\quad+\|\Psi^{k+1}-\Psi^{k}\|_\Theta^2\\
	&\overset{\cref{eq3_lemma3}}{=}\|\widetilde X^k\|_{\gamma^{-1}}^2+\|\widetilde \Psi^k\|_\Theta^2-\|X^{k+1}-X^{k}\|_{\gamma^{-1}}^2\\
	&\quad-2\langle \widetilde X^k,\nabla F^k-\nabla F^\star\rangle-2\langle \nabla F^k-\nabla F^\star, X^{k+1}-X^k\rangle\\
	&\overset{\cref{eq1_lemma9}}{\leq} \Big(1-\frac{2\mu L\gamma}{\mu+L}\Big)\|\widetilde X^k\|_{\gamma^{-1}}^2-\|\widetilde \Psi^{k+1}\|_{\gamma}^2\\
	&\quad+\|\widetilde \Psi^k\|_\Theta^2+\Big(1-\frac{2\gamma^{-1}}{\mu+L}\Big)\|\nabla F^k-\nabla F^\star\|_{\gamma}^2
	\ena
	The desired result is then obtained.
\end{proof}

\section{Proof of \cref{theo5}}\label{adx3}
We first show some properties of the norm $\|A\|_\text{max}\triangleq \max_{i}\|\va_{i*}\|_{\p}$.
\begin{lemma}\label{lemma7}
	\begin{enumerate}[label=(\alph*),noitemsep]
		\item For any $A\in\bR^{n\times d}$, we have $\|A\|_\text{max}\leq \bar{c}\|A\|$ and $\|A\|\leq \underline{c}\|A\|_\text{max}$, where $\bar{c}=d^{1/p-1/2}$ and $\underline{c}=\sqrt{n}$ for $p\in[1,2)$, and $\bar{c}=1$ and $\underline{c}=\sqrt{n}d^{1/2-1/p}$ for $p\geq 2$.
		\item For any two matrices $A$ and $B$ with compatible sizes, it holds that $\maxn{AB}\leq\|A\|_{\infty}\maxn{B}$.
	\end{enumerate}
\end{lemma}
\begin{proof}
	(a) The result follows from the definition of $p$-norm and H\"older's inequality \cite{boyd2004convex}.

	(b) The result holds since $\|AB\|_\text{max}= \max_{i}\|\va_{i*}B\|_{\p}\leq\max_{i}\sum_{j=1}^n\|a_{ij}\vb_{j*}\|_{\p}\leq\max_{i}\|\va_{i*}\|_1\max_j\|\vb_{j*}\|_{\p}=\|A\|_{\infty}\maxn{B}$.
\end{proof}

\begin{proof}[Proof of \cref{theo5}]
	We prove the result by mathematical induction.   It can be easily checked that \cref{eq1_theo5} holds for $k=0$. Suppose that 
	\bea
		\|\widehat X^k-X^k\|_{\text{max}}&\leq c_s\beta^k=s^k,\\
		\|X^k-\bar X\|&\leq \|X^0-\bone\bar\vx^\T\|\beta^k\leq\varsigma c_s\beta^k\label{mi2}
	\ena
	 for some $k\geq 0$, where $\bar X=\bone\bar\vx^\T$. We have
	\bea\label{eq2_theo5}
	&\maxn{\widehat X^{k+1}-X^{k}}=\maxn{s^k\sQ(\Delta^k/s^k)-\Delta^k}\\
	&=s^k\maxn{\sQ(\Delta^k/s^k)-\Delta^k/s^k}\leq s^k\delta=\delta c_s\beta^k
	\ena
	where the inequality follows from \cref{assum3} and \eqref{mi2}. Note that $\underline\beta<1$, it implies that
	\bea
	&\maxn{\widehat X^{k+1}-X^{k+1}}\\
	&=\maxn{\widehat X^{k+1}-X^k+\gamma(I-W)\widehat X^{k+1}}\\
	&=\maxn{(I+\gamma(I-W))(\widehat X^{k+1}-X^k)\\
		&\qquad+\gamma(I-W)(X^k-\bar X)}\\
	&\leq(1+c_W\gamma)\maxn{\widehat X^{k+1}-X^k}+c_W\bar{c}\gamma\|X^k-\bar X\|\\
	&\leq\big((1+2\gamma)\delta +2\bar{c}\gamma \varsigma\big) c_s\beta^k\\
	&=\Big(\frac{1}{2}+\delta\big(\frac{1+2\gamma}{2}+\frac{2\underline{c}\bar{c}\gamma}{\rho}\big)\Big) c_s\beta^k\leq c_s\beta^{k+1}
	\ena
	where $c_W=1+\|W\|_{\infty}= 2$. The first inequality used \cref{lemma7} and \cref{assum2}, and the second inequality follows from \eqref{mi2} and \cref{eq2_theo5}.

	Moreover, we have
	\bea\label{eq3_theo5}
	&\twon{X^{k+1}-\bar X}\\
	&\overset{\cref{alg_con2}}{=}\twon{X^k-\bar X+\gamma(W-I)(\widehat X^{k+1}-X^k)\\
		&\qquad+\gamma(W-I)(X^k-\bar X)}\\
	&=\twon{(I+\gamma(W-I))(X^k-\bar X)+\gamma(W-I)(\widehat X^{k+1}-X^k)}\\
	&\leq(1-\gamma\rho)\twon{X^k-\bar X}+2\gamma \underline{c}\maxn{\widehat X^{k+1}-X^k}
	\ena
	where the inequality follows from \cref{lemma3,lemma7}.

	Let $\underline{\varsigma}={2\delta \underline{c}}/{\rho}\geq 0$ and $\bar \varsigma=\frac{1-(1+2\gamma)\delta}{2\gamma}>0$. We have $\underline{\varsigma}\leq\bar \varsigma$ and
	\bee\label{eq4_theo5}
	\varsigma=\frac{1}{2}(\underline{\varsigma}+\bar\varsigma)=\frac{\delta \underline{c}}{\rho}+\frac{1-(1+2\gamma)\delta}{4\gamma}.
	\ene

	By \eqref{mi2}, it follows from \cref{eq3_theo5} that
	\bea
	\twon{X^{k+1}-\bar X}&\leq\Big((1-\gamma\rho)+{2\gamma \underline{c}\delta}/{\varsigma}\Big)\varsigma c_s\beta^{k}\\
	&\leq\Big((1-\gamma\rho)+{\gamma \underline{c}\delta}/{\underline\varsigma}+{\gamma \underline{c}\delta}/{\bar\varsigma}\Big)\varsigma c_s\beta^{k}\\
	&=\Big(1-\frac{\gamma\rho}{2}+\frac{2\delta\gamma^2  \underline{c}\bar{c}}{1-(1+2\gamma)\delta}\Big)\varsigma c_s\beta^{k}\\
	&\leq\varsigma c_s\beta^{k+1}
	\ena
	where the second inequality used the relation $\frac{1}{\varsigma}\leq\frac{1}{2\underline{\varsigma}}+\frac{1}{2\bar{\varsigma}}$. 
	Thus, we have shown that $\twon{X^{k+1}-\bar X}\leq \|X^0-\bone\bar\vx^\T\|\beta^{k+1}$ and $\maxn{\widehat X^{k+1}-X^{k+1}}\leq c_s\beta^{k+1}$. Thus, he result holds by using $c_p=\bar{c}\underline{c}\leq\sqrt{n}d^{|\frac{1}{2}-\frac{1}{\vp}|}$ from \cref{lemma7}.
\end{proof}

\section{Proof of \cref{theo3}}\label{adx4}
\subsection{Preliminary lemmas}
\begin{lemma}\label{lemma10}
	Under \cref{assum4,assum2,assum3}. Let $\{X^k\}$ and $\{\Psi^k\}$ be generated by \cref{alg_opt2}. The following relation holds,
	\bea\label{eq1_lemma10}
	&\|X^{k+1}-X^\star\|_{\gamma^{-1}}^2+\|\Psi^{k+1}-\Psi^\star\|_{\Theta+\gamma I}^2\\
	&\quad-2\langle\Psi^{k+1}-\Psi^\star,\widehat Y^{k+1}-Y^k\rangle+\|\Psi^{k+1}-\Psi^{k}\|_\Theta^2\\
	&\leq \Big(1-\frac{2\mu L\gamma}{\mu+L}\Big)\|X^{k}-X^\star\|_{\gamma^{-1}}^2+\|\Psi^{k}-\Psi^\star\|_\Theta^2\\
	&\quad+\Big(1-\frac{2\gamma^{-1}}{\mu+L}\Big)\|\nabla F^k-\nabla F^\star\|_{\gamma}^2,\forall k\geq 0
	\ena
\end{lemma}
\begin{proof}
	The proof follows from the same line of that of \cref{lemma9}, and hence we omit it here.
\end{proof}

\begin{lemma}\label{lemma4}
	Suppose \cref{assum4,assum2,assum3} hold. Then, for any $\gamma\in(0,\frac{2}{\mu+L})$ and $\zeta\in(0,\gamma)$, the following relation holds,
	\bea\label{eq1_lemma4}
	&\|X^{k+1}-X^\star\|_{\gamma^{-1}}^2+\|\Psi^{k+1}-\Psi^\star\|_{\Theta+\gamma I-\zeta I}^2+\|\Psi^{k+1}-\Psi^k\|_\Theta^2\\
	&\leq \nu\Big(\|X^{k}-X^\star\|_{\gamma^{-1}}^2+\|\Psi^{k}-\Psi^\star\|_{\Theta+\gamma I-\zeta I}^2\Big)\\
	&\quad+\zeta^{-1} \underline{c}\|\widehat Y^{k+1}-Y^k\|_\text{max}^2,\ \forall k\geq0,
	\ena
	where
	\bee\label{eq1_lemma3}
	\nu\triangleq\max\Big\{1-\frac{2\mu L\gamma}{\mu+L},1-\frac{(\gamma-\zeta)\tau(1-\lambda_2(W))}{1-\zeta\tau(1-\lambda_2(W))}\Big\}<1.
	\ene
\end{lemma}
\begin{proof}
	It follows from \cref{eq1_lemma10} in \cref{lemma10}, $\gamma\leq\frac{2}{\mu+L}$, and the Cauchy-Schwarz inequality that
	\bea\label{eq2_lemma10}
	&\|\widetilde X^{k+1}\|_{\gamma^{-1}}^2+\|\widetilde \Psi^{k+1}\|_{\Theta+\gamma I}^2+\|\Psi^{k+1}-\Psi^k\|_\Theta^2\\
	&\leq \big(1-\frac{2\mu L\gamma}{\mu+L}\big)\|\widetilde X^k\|_{\gamma^{-1}}^2+\|\widetilde \Psi^k\|_{\Theta}^2\\
	&\quad+\zeta\|\widetilde \Psi^{k+1}\|^2+\zeta^{-1}\|\widehat Y^{k+1}-Y^k\|^2.
	\ena
	Since $\zeta\in(0,\gamma)$, we have
	\bea
	\Theta&\preceq \frac{\tau^{-1}(1-\lambda_2(W))^{-1}-\gamma}{\tau^{-1}(1-\lambda_2(W))^{-1}-\zeta} (\Theta+\gamma I-\zeta I)\\
	&\preceq\Big(1-\frac{(\gamma-\zeta)\tau(1-\lambda_2(W))}{1-\zeta\tau(1-\lambda_2(W))}\Big)(\Theta+\gamma I-\zeta I).
	\ena
	It then follows from \cref{eq2_lemma10} and \cref{lemma7} that
	\bea
	&\|\widetilde X^{k+1}\|_{\gamma^{-1}}^2+\|\widetilde \Psi^{k+1}\|_{\Theta+\gamma I-\zeta I}^2+\|\Psi^{k+1}-\Psi^k\|_\Theta^2\\
	&\leq \big(1-\frac{2\mu L\gamma}{\mu+L}\big)\|\widetilde X^k\|_{\gamma^{-1}}^2+\zeta^{-1} \underline{c}\|\widehat Y^{k+1}-Y^k\|_\text{max}^2\\
	&\quad+\Big(1-\frac{(\gamma-\zeta)\tau(1-\lambda_2(W))}{1-\zeta\tau(1-\lambda_2(W))}\Big)\|\widetilde \Psi^k\|_{\Theta+\gamma I-\zeta I}^2
	\ena
	which is the desired result \cref{eq1_lemma4}.

\end{proof}

\begin{lemma}\label{lemma11}
	Suppose \cref{assum4,assum2,assum3} hold. Let $\tau>0$ and $\widetilde\tau\in(0,1)$. If $\gamma\in\big(0,\min\{\frac{2}{\mu+L},\frac{1}{L}\sqrt{\frac{\widetilde\tau}{1-\widetilde\tau}},\frac{\widetilde\tau}{\tau(1-\lambda_n(W))}\}\big)$, then
	\bea\label{eq1_lemma11}
	&\|\widehat Y^{k+1}-Y^{k+1}\|_{\text{max}}^2\\
	&\leq\frac{8\gamma\widetilde\tau\Big(\|\widetilde \Psi^{k+1}\|_{\Theta+\frac{\gamma}{2} I}^2+\|\widetilde X^{k+1}\|_{\gamma^{-1}}^2+\|\Psi^k-\Psi^{k+1}\|_{ \Theta}^2\Big)}{(1-\delta)^2(1-\widetilde\tau)} \\
	&\quad+\frac{2}{1+\delta}\|\widehat Y^{k+1}-Y^{k}\|_{\text{max}}^2\\
	\ena
\end{lemma}
\begin{proof}
	Let $\xi_1,\xi_2,\xi_3$ be some positive numbers. We have
	\bea
	&\|\widehat Y^{k+1}-Y^{k+1}\|_\text{max}^2\\
	&\overset{\cref{alg_opt2}}{=}\|\widehat Y^{k+1}-Y^k-\gamma(\Psi^k-\Psi^{k+1})\\
	&\quad+\gamma(\Psi^{k+1}-\Psi^\star)+\gamma(\nabla F^{k+1}-\nabla F^\star)\|_\text{max}^2\\
	&\leq (1+\xi_1)\|\widehat Y^{k+1}-Y^k\|_\text{max}^2+(1+\xi_1^{-1})(1+\xi_2)\bar{c}\|\widetilde \Psi^{k+1}\|_{\gamma^2}^2\\
	&\quad+(1+\xi_1^{-1})(1+\xi_2^{-1})\bar{c}\gamma^2\times\\
	&\quad\Big((1+\xi_3)\|\nabla F^{k+1}-\nabla F^\star\|^2+(1+\xi_3^{-1})\|\Psi^k-\Psi^{k+1}\|^2\Big)\\
	&\overset{\cref{eq_lipschitz}}{\leq}(1+\xi_1)\|\widehat Y^{k+1}-Y^k\|_\text{max}^2+ (1+\xi_1^{-1})(1+\xi_2)\bar{c}\gamma\Big(\|\widetilde \Psi^{k+1}\|_{\gamma}^2\\
	&\;+\xi_2^{-1}\gamma\big((1+\xi_3)L^2\|\widetilde X^{k+1}\|^2+(1+\xi_3^{-1})\|\Psi^k-\Psi^{k+1}\|^2\big)\Big)\\
	\ena
	where the first inequality used the Cauchy-Schwarz inequality three times. Set $\xi_1=\xi_2={(1-\delta)}/{(1+\delta)}>0$ and $\xi_3=\left({\gamma L^2(\tau^{-1}(1-\lambda_n(W))^{-1}-\gamma)}\right)^{-1}>0$. We have
	\bea\label{eq2_lemma11}
	&\|\widehat Y^{k+1}-Y^{k+1}\|_\text{max}^2\\
	&\leq(1+\xi_1)\|\widehat Y^{k+1}-Y^k\|_\text{max}^2+ (1+\xi_1^{-1})(1+\xi_2)\bar{c}\gamma\Big(\|\widetilde \Psi^{k+1}\|_{\gamma}^2\\
	&\quad+\frac{\xi_2^{-1}\gamma(1+\xi_3^{-1})\big(\|\widetilde X^{k+1}\|_{\gamma^{-2}}^2+\|\Psi^k-\Psi^{k+1}\|_{\psi-1}^2\big)}{\psi-1}\Big)\\
	&\leq\frac{2\|\widehat Y^{k+1}-Y^k\|_\text{max}^2}{1+\delta}\\
	&\quad+ (1+\xi_1^{-1})(1+\xi_2)\bar{c}\gamma^2\Big((\psi-\frac{1}{2})\|\widetilde \Psi^{k+1}\|_{\gamma (\Theta+\frac{\gamma}{2} I)}^2\\
	&\quad+\frac{\xi_2^{-1}(1+\xi_3^{-1})\big(\|\widetilde X^{k+1}\|^2+\|\Psi^k-\Psi^{k+1}\|_{\gamma \Theta}^2\big)}{\gamma(\psi-1)}\Big)\\
	&\leq\frac{2\|\widehat Y^{k+1}-Y^{k}\|_{\text{max}}^2}{1+\delta}\\
	&\quad+ \bar{c}\gamma \vartheta \Big(\|\widetilde \Psi^{k+1}\|_{\Theta+\frac{\gamma}{2} I}^2+\|\widetilde X^{k+1}\|_{\gamma^{-1}}^2+\|\Psi^k-\Psi^{k+1}\|_{ \Theta}^2\Big)\\
	\ena
	where $\psi=\tau^{-1}\gamma^{-1}(1-\lambda_n(W))^{-1}$ and
	\bea
	\vartheta&=\gamma\max\Big\{(1+\xi_1^{-1})(1+\xi_2)\gamma(\psi-\frac{1}{2}),\\
	&\qquad\qquad\quad\frac{(1+\xi_1^{-1})(1+\xi_2^{-1})(1+\xi_3^{-1})}{\tau^{-1}(1-\lambda_n(W))^{-1}-\gamma}\Big\}\\
	&=\max\Big\{\frac{4\gamma(\tau^{-1}(1-\lambda_n(W))^{-1}-\frac{\gamma}{2})}{1-\delta^2},\\
	&\qquad\qquad\frac{4\gamma^2 L^2}{(1-\delta)^2}+\frac{4\gamma\tau(1-\lambda_n(W))}{(1-\delta)^2(1-\gamma\tau(1-\lambda_n(W)))}\Big\}\\
	&\leq\frac{8\widetilde\tau}{(1-\delta)^2(1-\widetilde\tau)}
	\ena
	and the last inequality follows from $\gamma\leq\frac{1}{L}\sqrt{\frac{\widetilde\tau}{1-\widetilde\tau}}$. The desired result follows immediately.
\end{proof}

\subsection{Proof of \cref{theo3}}
\begin{proof}
	We prove the result by mathematical induction. Suppose that $\|\widehat Y^{k}-Y^{k}\|_\text{max}^2\leq c\beta^k$ and $\|X^{k}-X^\star\|_{\gamma^{-1}}^2+\|\Psi^{k}-\Psi^\star\|_{\Theta+\frac{\gamma}{2} I}^2\leq \widetilde c\beta^k$. Clearly, it holds when $k=0$. Then, we have from \cref{assum3} that
	\bee\label{eq3_lemma10}
	\|\widehat Y^{k+1}-Y^k\|_\text{max}^2=\|s^k\sQ(\varDelta^k/s^k)-\varDelta^k\|_\text{max}^2\leq \delta s^k=\delta c\beta^k.
	\ene
	We have from \cref{lemma4,lemma11} that
	\bea\label{eq2_theo3}
	&\|X^{k+1}-X^\star\|_{\gamma^{-1}}^2+\|\Psi^{k+1}-\Psi^\star\|_{\Theta+\frac{\gamma}{2} I}^2\\
	&\overset{\cref{eq1_lemma4}}{\leq}\nu\widetilde c\beta^k+2\gamma^{-1} \underline{c}\delta c\beta^k
	\ena
	and
	\bea\label{eq1_theo3}
	&\|\widehat Y^{k+1}-Y^{k+1}\|_\text{max}^2\\
	&\overset{\cref{eq1_lemma11},\cref{eq1_lemma4}}{\leq}\frac{8\bar{c}\gamma\widetilde\tau\nu\widetilde c\beta^k}{(1-\delta)^2(1-\widetilde\tau)}+\Big(\frac{16\widetilde\tau \underline{c}\bar{c}}{(1-\delta)^2(1-\widetilde\tau)}+\frac{2}{1+\delta}\Big)\delta c\beta^k
	\ena
	where we have set $\zeta=\frac{\gamma}{2}$ and used the hypothesis and \cref{eq3_lemma10}.

	Define the following two functions:
	\bea
	g_1(t)\triangleq\nu+2\gamma^{-1} \underline{c}\delta t\\
	\ena
	\bea
	g_2(t)\triangleq\frac{2\delta}{1+\delta}+\frac{16\delta\widetilde\tau \underline{c}\bar{c}}{(1-\delta)^2(1-\widetilde\tau)}+\frac{8\bar{c}\gamma\widetilde\tau\nu}{t(1-\delta)^2(1-\widetilde\tau)}.
	\ena
	In view of \cref{eq2_theo3,eq1_theo3}, it is sufficient to show that $g_1(\varsigma)\leq\beta$ and $g_2(\varsigma)\leq\beta$ since $\varsigma=c/\widetilde c$. To this end, notice that $0<\varsigma_1<\varsigma_2$, since
	\bea
	&\Big(\frac{(1-\delta)^3(1-\widetilde\tau)}{1+\delta}-16\delta\widetilde\tau \underline{c}\bar{c}\Big)(\varsigma_1-\varsigma_2)\\
	&=8\bar{c}\gamma\widetilde\tau-\frac{\gamma(1-\delta)^3(1-\widetilde\tau)(1-\nu)}{2\delta(1+\delta) \underline{c}}\leq 0
	\ena
	where we used $\widetilde\tau\leq \frac{1}{2}$ and the bounds of $\gamma$ and $\tau$ in \cref{theo3}.

	An important relation is that $g_1(\varsigma_2)=1$ and $g_2(\varsigma_1)=1$, which can be readily checked. Moreover, $g_1$ is strictly increasing and $g_2$ is strictly decreasing w.r.t. $\varsigma$, and both are convex functions. Therefore, we have
	\bee
	g_1(\varsigma)<\frac{g_1(\varsigma_1)+g_1(\varsigma_2)}{2}\leq\frac{1+\beta_1}{2}\leq\beta
	\ene
	and
	\bee
	g_2(\varsigma)<\frac{g_2(\varsigma_1)+g_2(\varsigma_2)}{2}\leq\frac{1+\beta_2}{2}\leq\beta
	\ene
	where $\beta_1=g_1(\zeta_1)=\frac{(1-\delta)^3(1-\widetilde\tau)\nu}{(1-\delta)^3(1-\widetilde\tau)-16\delta\widetilde\tau(1+\delta) c_{p}}<1$, $\beta_2=g_2(\zeta_2)=\frac{2\delta}{1+\delta}+\frac{16\delta\widetilde\tau c_{p}}{(1-\delta)^2(1-\nu)(1-\widetilde\tau)}<1$, and we used $c_p=\bar{c}\underline{c}\leq\sqrt{n}d^{|\frac{1}{2}-\frac{1}{\vp}|}$. Therefore, we finished the induction.
\end{proof}

\begin{IEEEbiography}
	[{\includegraphics[width=1in,height=1.25in,clip,keepaspectratio]{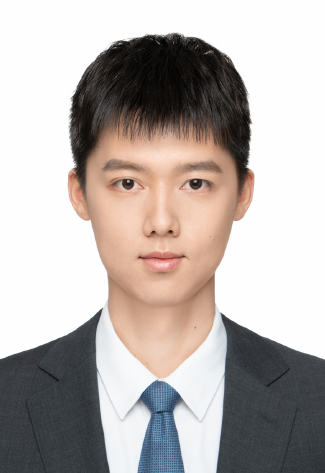}}]
	{Jiaqi Zhang} received the B.S. degree in electronic and information engineering from the School of Electronic and Information Engineering, Beijing Jiaotong University, Beijing, China, in 2016. He is currently pursuing the Ph.D. degree at the Department of Automation, Tsinghua University, Beijing, China. His current research interests include networked control systems, distributed optimization and learning, and their applications.
\end{IEEEbiography}
\begin{IEEEbiography}
	[{\includegraphics[width=1in,height=1.25in,clip,keepaspectratio]{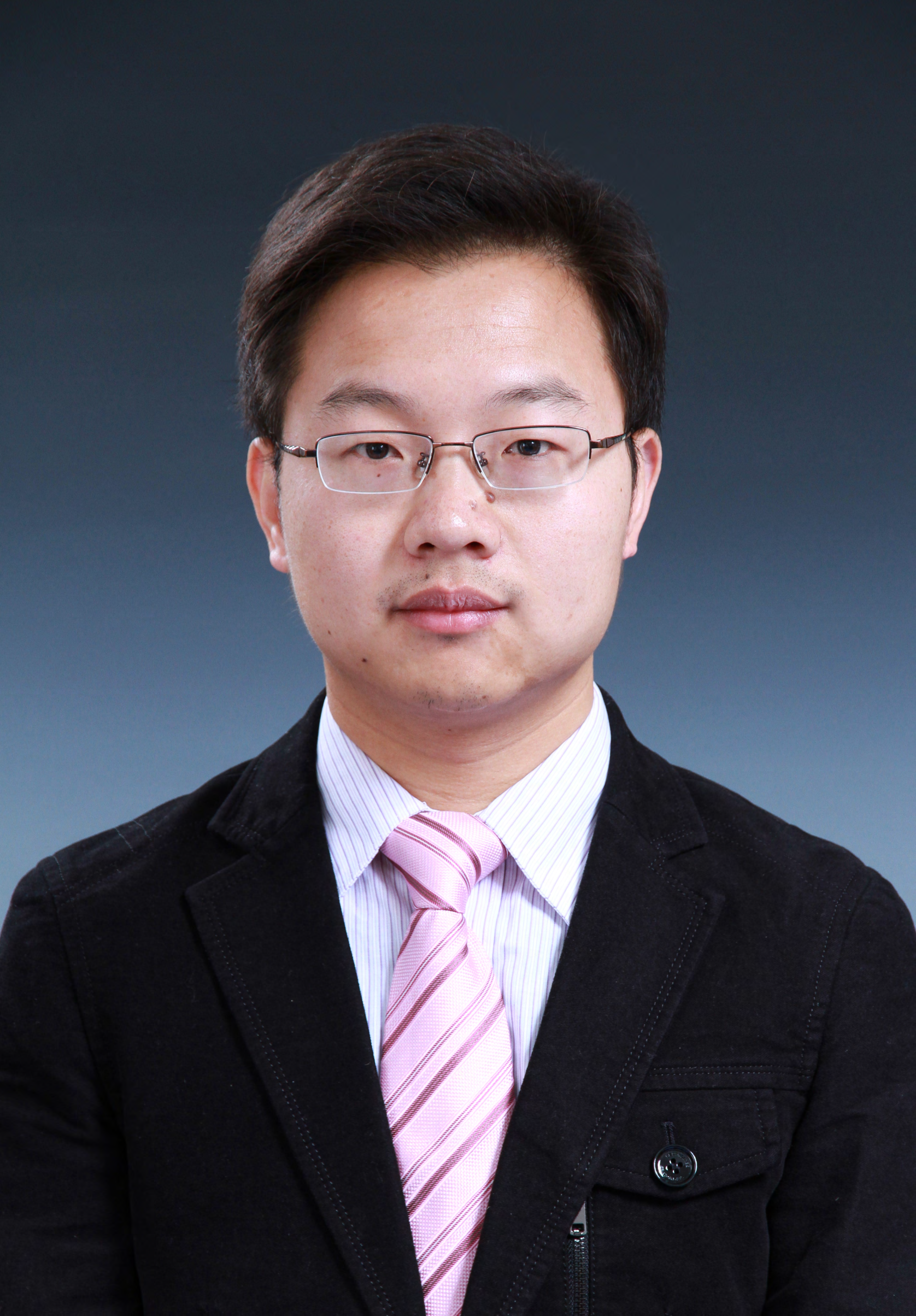}}]
	{Keyou You} (SM'17) received the B.S. degree in Statistical Science from Sun Yat-sen University, Guangzhou, China, in 2007 and the Ph.D. degree in Electrical and Electronic Engineering from Nanyang Technological University (NTU), Singapore, in 2012. After briefly working as a Research Fellow at NTU, he joined Tsinghua University in Beijing, China where he is now a tenured Associate Professor in the Department of Automation. He held visiting positions at Politecnico di Torino, Hong Kong University of Science and Technology, University of Melbourne and etc. His current research interests include networked control systems, distributed optimization and learning, and their applications.

	Dr. You received the Guan Zhaozhi award at the 29th Chinese Control Conference in 2010 and the ACA (Asian Control Association) Temasek Young Educator Award in 2019. He received the National Science Fund for Excellent Young Scholars in 2017. He is serving as an Associate Editor for the IEEE Transactions on Control of Network Systems, IEEE Transactions on Cybernetics, IEEE Control Systems Letters(L-CSS), Systems \& Control Letters.
\end{IEEEbiography}

\begin{IEEEbiography}
	[{\includegraphics[width=1in,height=1.25in,clip,keepaspectratio]{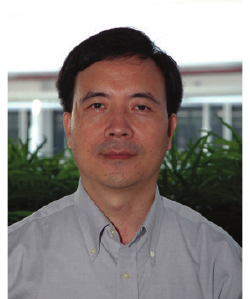}}]
	{Lihua Xie} (F'07) received the B.E. and M.E. degrees in electrical engineering from Nanjing University of Science and Technology in 1983 and 1986, respectively, and the Ph.D. degree in electrical engineering from the University of Newcastle, Australia, in 1992. Since 1992, he has been with the School of Electrical and Electronic Engineering, Nanyang Technological University, Singapore, where he is currently a professor and Director, Delta-NTU Corporate Laboratory for Cyber-Physical Systems. He served as the Head of Division of Control and Instrumentation from July 2011 to June 2014. He held teaching appointments in the Department of Automatic Control, Nanjing University of Science and Technology from 1986 to 1989 and Changjiang Visiting Professorship with South China University of Technology from 2006 to 2011.

	His research interests include robust control and estimation, networked control systems, multiagent networks, localization, and unmanned systems. He is an Editor-in-Chief for Unmanned Systems and has served as the Editor of IET Book Series in Control and an Associate Editor for a number of journals including the IEEE Transactions on Automatic Control, Automatica, the IEEE Transactions on Control Systems Technology, IEEE Transactions on Network Control Systems, and the IEEE Transactions on Circuits and Systems-II. He was the IEEE Distinguished Lecturer from January 2012 to December 2014 and an Elected Member of Board of Governors, the IEEE Control System Society from January 2016 to December 2018. He is a Fellow of IFAC and Fellow of Academy of Engineering Singapore.	
\end{IEEEbiography}

\end{document}